\documentclass[11pt]{article}
\usepackage[T1]{fontenc}
\usepackage{amsfonts}
\usepackage{mathrsfs}
\usepackage{amsmath}
\usepackage{amssymb}
\usepackage{amsmath}
\usepackage{float}
\usepackage{booktabs}
\usepackage{multirow}
\usepackage{ifpdf}
\usepackage{enumerate}
\usepackage{graphicx}
\usepackage{color}
\usepackage{lineno}
\usepackage{caption}
\usepackage{bm}
\newcommand{\R}{{\mathbb R}}

\newcommand{\ML}{{\mathcal L}}

\newcommand{\be}{\begin{eqnarray}}
\newcommand{\ben}{\begin{eqnarray*}}
\newcommand{\en}{\end{eqnarray}}
\newcommand{\enn}{\end{eqnarray*}}

\newtheorem{theorem}{Theorem}[section]
\newtheorem{lemma}[theorem]{Lemma}
\newtheorem{remark}[theorem]{Remark}

\newtheorem{definition}[theorem]{Definition}

\newtheorem{proposition}[theorem]{Proposition}

\newenvironment{proof}{{\em Proof.}}{$\Box$}
\definecolor{rot}{rgb}{0.000,0.000,0.000}
\newcommand{\tcr}{\textcolor{rot}}
\begin{document}
\renewcommand{\theequation}{\arabic{section}.\arabic{equation}}
\begin{titlepage}
\title{\bf Time-harmonic acoustic scattering from a non-locally perturbed trapezoidal surface
}

\author{
Wangtao Lu\thanks{ School of Mathematical Sciences, Zhejiang University, Hangzhou 310027, China ({\sf wangtaolu@zju.edu.cn}).}\quad \quad
Guanghui Hu\thanks{Beijing Computational Science Research Center, Beijing 100193, China ({\sf hu@csrc.ac.cn}).} \qquad
}
\date{}
\end{titlepage}
\maketitle

%\vspace{.2in}

\begin{abstract}
This paper is concerned with acoustic scattering  from a sound-soft
 \tcr{trapezoidal surface} in two dimensions. The  trapezoidal surface is supposed to
  consist of two horizontal half-lines pointing oppositely, and a
    single finite vertical line segment connecting their endpoints, which can be regarded
  as a \tcr{non-local} perturbation of a  straight
  line. For incident plane waves, we enforce that the
    scattered wave, post-subtracting reflected plane waves by the two half lines
    of the scattering surface in certain two regions respectively, satisfies an
    integral form of Sommerfeld radiation condition at infinity. With this new
    radiation condition, we prove uniqueness and existence of weak solutions by a
  coupling scheme between finite element and integral equation methods.
   This consequently indicates that our new radiation condition is
    sharper than the Angular Spectrum Representation, and has generalized the
  radiation condition for scattering problems in a locally perturbed half-plane.

 Furthermore, we develop a numerical mode matching method  based on this new radiation condition. A perfectly matched layer is
    setup to absorb outgoing waves at infinity. Since the medium composes of two horizontally uniform
    regions, we expand, in either uniform region, the scattered wave in terms of
    eigenmodes and match the mode expansions on the common interface between the
    two uniform regions, which in turn gives rise to numerical solutions to our
    problem. Numerical experiments are carried out to validate the new radiation
    condition and to show the performance of our numerical method.
    
\vspace{.2in} {\bf Keywords:} Helmholtz equation, trapezoidal surface, Sommerfeld radiation condition, half plane,  non-local perturbation, mode matching method, perfectly matched layer.

\end{abstract}

\section{Introduction}

Wave scattering in a layered medium and \tcr{in a half plane} has numerous
applications in scientific and engineering areas \cite{chew95}. \tcr{For the
  purpose of proving} well-posedness of the scattering model, understanding the
physical radiation behavior of the wave field at infinity is critical. \tcr{A
  rigorous description of the asymptotics} not only helps to design a proper
radiation condition for the wave field at infinity, but also helps to truncate
the unbounded domain with an accurate boundary condition for numerically solving
the problem.

  Certainly, the radiation condition is structure-related. For instances, if
  \tcr{a bounded obstacle is embedded into a homogeneous background medium}, the
  scattered wave is purely outgoing at infinity and satisfies the classic
  Sommerfeld radiation condition (SRC) \cite{colkre13}. When the structure is
  filled in by a two-layered medium with a locally perturbed planar surface, the
  perturbed wave field due to the local perturbation, is outgoing at infinity
  and satisfies the SRC \cite{roazha92, chezhe10, baohuyin18, Li2010}; see also
  \cite{Bao00, Wood2006, Jin98, Willers1987} for studies on impenetrable locally
  perturbed surfaces. However, in the case of a globally perturbed rough
  surface, by which we mean a non-local perturbation of a planar surface such
  that the surface lies within a finite distance of the original plane, one in
  general cannot explicitly extract an outgoing wave from the scattered wave to
  meet the  SRC. \tcr{The Angular Spectrum
    Representation (ASR) radiation condition (see \cite{J.A, S.P, CWEL}) or the
    Upward Propagating Radiation Condition (UPRC) (see \cite{ZC98, ZC2003,
      CHP2006}) can be viewed as a rigorous formulation of a radiation condition
    to show the well-posedness of the problem in both two and three dimensions.
    The radiation condition relies also on the type of incoming waves in rough
    surface scattering problems. The authors in \cite{CWEL} proved that the
    scattered field incited by a plane wave decays slower than that for a point
    source wave in the horizontal direction. It was recently proved in
    \cite{Hu2018} that the scattered field due to a point source wave fulfills
    the Sommerfeld radiation condition in a half plane, which however does not
    hold true for plane wave incidence.}

  In this paper, we shall study a special class of two-dimensional (2D) rough
  surface scattering problems, where the globally perturbed surface is assumed
  to consist of two horizontal half lines pointing oppositely and one single
  vertical line segment connecting their endpoints; we shall propose a novel
  SRC-type radiation condition that is sharper than ASR. \tcr{Relying on this
    new radiation condition}, we shall prove existence and uniqueness of weak
  solutions and design a numerical method for the scattering problem.

Let $\Gamma$ be a sound-soft rough surface in two dimensions and suppose that
the region above $\Gamma$, which we denote by $\Omega_\Gamma$, is occupied by a
homogeneous and isotropic medium. Consider a time-harmonic plane wave
$u^{in}=e^{ikx\cdot d}$ incident onto {\color{rot} the} rough surface $\Gamma\subset \R^2$ from
above. Here, $k>0$ denotes the wave number, $d=(\cos\theta, -\sin\theta)$
stands for the incident direction and $\theta\in(0,\pi)$ is the incident angle
with the positive $x_1$-axis. In this paper, we suppose that $\Gamma$ consists
of three parts (see Fig.~\ref{fig:trap_pec}):
 \begin{equation}
   \label{eq:Gamma}
\Gamma=\{(x_1,0)|x_1\leq 0\}\cup V_h \cup \{(x_1,-h)|x_1\geq 0\}
\end{equation}
where $V_h=\{(0,x_2)|-h\leq x_2\leq 0\}$ \tcr{is a finite vertical line segment with the height $h>0$}.
\begin{figure}[!ht]
  \centering
  \includegraphics[width=0.5\textwidth, viewport=0.pt 130.pt 1000.pt
  740.pt,clip  ]{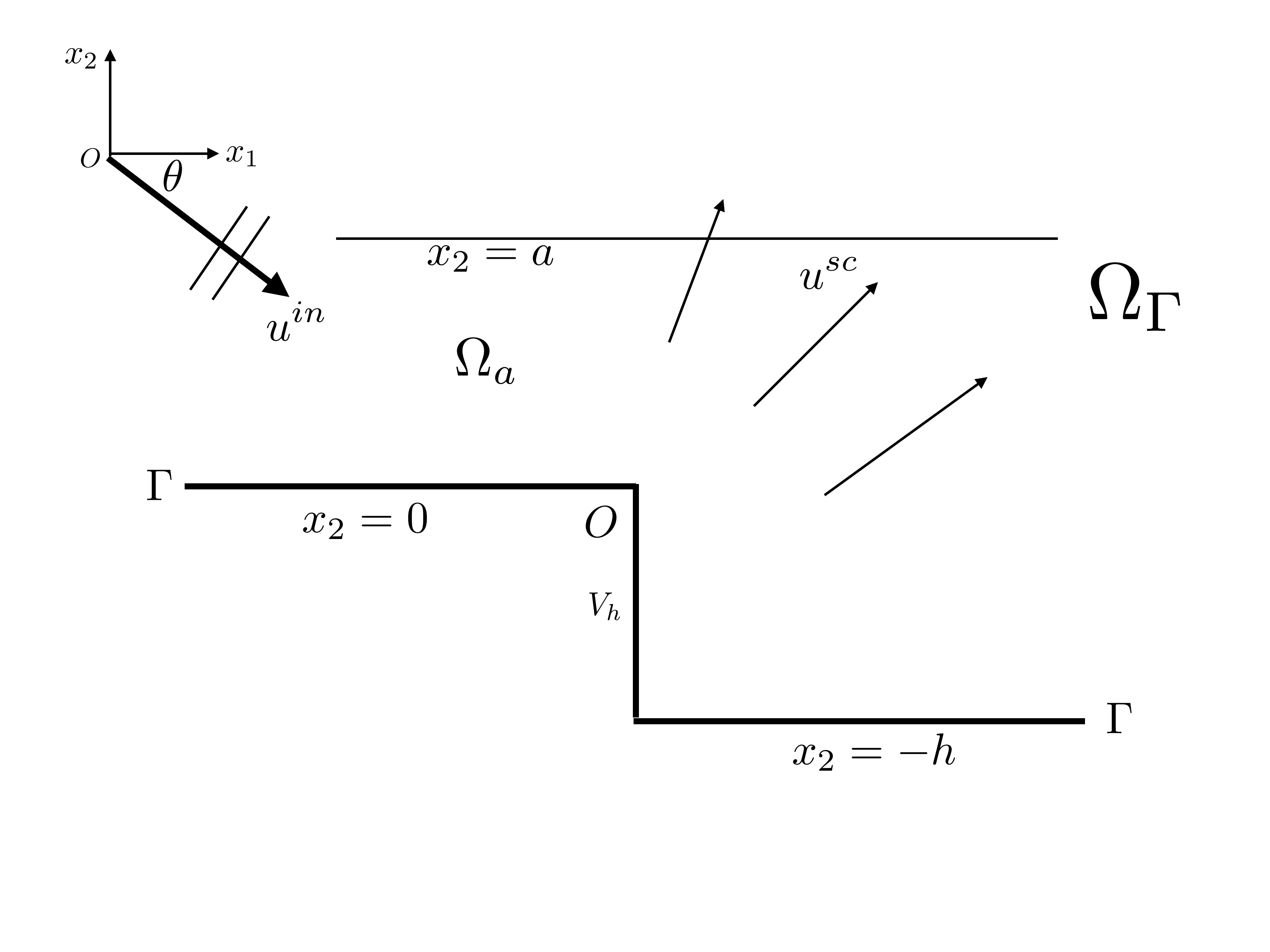}
  \caption{An incoming plane wave with the angle $\theta\in(0,\pi/2)$ is
    incident onto a trapezoidal surface $\Gamma$ in a half plane. {\color{rot}$\Omega_\Gamma$: the domain above $\Gamma$. $\Omega_a$: the domain between $\Gamma$ and $x_2=a$}.}
  \label{fig:trap_pec}
\end{figure}
Such kind of \tcr{trapezoidal surfaces is a non-local perturbation of the straight line $\{x_2=0\}$, which can be regarded as a special kind of globally perturbed rough surfaces}. Since the region $\Omega_\Gamma$ fulfills the geometrical condition
\be\label{shape-assumption}
(x_1, x_2)\in \Omega_\Gamma\quad \rightarrow\quad (x_1, x_2+s)\in \Omega_\Gamma\quad\mbox{for all}\quad s>0,
\en
 by \cite{CWEL} there exists a unique solution ${\color{rot}u^{tol}}\in H_\varrho^1({\color{rot}\Omega_a})$ for any $a>0$ and $\varrho\in(-1,-1/2)$ such that
\be
  \label{eq:hlm}
\Delta u^{tot}+k^2u^{tot}=0\quad&&\mbox{in}\quad \Omega_\Gamma,\quad u^{tot}=u^{in}+u^{sc},\\
  \label{eq:bc}
u^{tot}=0\quad&&\mbox{on}\quad \Gamma.
\en
Here, ${\color{rot}\Omega_a}=\{x\in \Omega_\Gamma, x_2<a\}$ is the unbounded strip between $\Gamma$ and $x_2=a$ (see Figure \ref{fig:trap_pec}), and
$H^{1}_\varrho({\color{rot}\Omega_a})$ stands for
weighted Sobolev space defined as
\be\label{WSS}
 \|u\|_{H^1_{\varrho}({\color{rot} \Omega_a })}
= \left[ \int_{{\color{rot} \Omega_a }} \bigg( \big| (1+|x_1|^2)^{\varrho/2}
      u\big|^2 + \left| \nabla \left[(1+|x_1|^2)^{\varrho/2} u\right]\right|^2 \bigg) d x
    \right]^{1/2} \, .
   \en
One can also employ the following equivalent norm to $||\cdot||_{H^1_{\varrho}({\color{rot}\Omega_a })}$:
\ben
||u||':=\left[ \int_{{\color{rot} \Omega_a }} (1+|x_1|^2)^{\varrho}\bigg( \big|
      u\big|^2 + \big| \nabla  u|^2 \bigg) d x
    \right]^{1/2} \, , \quad u\in H^1_{\varrho}({\color{rot} \Omega_a }).
\enn
Further, the scattered field $u^{sc}=u^{tot}-u^{in}$ in $\Omega_\Gamma$ is required to fulfill the {\color{rot} Upward  Angular Spectrum Representation (UASR)}, which can
be written as
\be\label{ASP}
u^{sc}(x)=\frac{1}{\sqrt{2 \pi}}\int_{\R}
\exp(i[(x_2-h)\sqrt{k^{2}-\xi^{2}}+x_1
\xi])\, \hat{u}^{sc}_{a}(\xi)\,d\xi, \quad x_1>a,
\en
where $\hat{u}^{sc}_{a}(\xi)$ denotes the Fourier transform of
$u^{sc}(x_1,a)$ in $x_1$, given by
\ben
\hat{u}^{sc}_{a}(\xi)=\frac{1}{\sqrt{2 \pi}}\int_{\R}\exp(-i x_1\, \xi)u^{sc}(x_1,a)\,dx_1.
\enn
 In this
equation $\sqrt{k^{2}-\xi^2}=i\sqrt{\xi^2-k^{2}}$ when $\xi^2>k^2$.
The representation of $u^{sc}$ in the integral (\ref{ASP}) can be
interpreted as a formal radiation condition in the physics and
engineering literature on rough surface scattering {\color{rot}(see e.g. \cite{S.P, J.A})}.

If $\Gamma$ is a local perturbation of the original straight line $\{x_2=0\}$ ,
it was proved in a series of papers {\color{rot}(see \cite{Bao00, baohuyin18, Jin98, Li2010, Willers1987, Wood2006})} that $u^{sc}$ can be decomposed into two
parts: $u^{sc}=u^{re}+v$ where $u^{re}$ is the scattered field corresponding to
unperturbed surface $\{x_2=0\}$ and $v\in H^1(\Omega_\Gamma)$ is caused by the
local perturbation satisfying the half-plane Sommerfeld radiation condition
\be
\label{RC}
\lim_{r\rightarrow\infty}\int_{S_{r}}|\partial_r v-ik v|^2\,ds\rightarrow0,\qquad
\sup_{r>0}
\int_{S_r}|v|^2\,ds<\infty, \en
where $S_r:=\{x\in
\Omega_\Gamma: |x|=r\}$. Physically, $u^{re}=-e^{ikx\cdot d'}$ with $d':=(d_1,-d_2)$ is uniquely
determined by Snell's law  and (\ref{RC}) indicates that $v$
  approximately becomes purely outgoing at infinity. We note that
  both $u^{re}$ and $v$ fulfill the ASR (\ref{ASP}), so that $u^{sc}$ also satisfies
   the ASR.  The decomposition of $u^{sc}$
into the sum of $u^{re}$ and $v$ provides a deep insight into the wave
phenomenon in a locally perturbed half plane, which however cannot apply to our
scattering problem in a non-locally perturbed half-plane. In this paper, we shall propose a new radiation condition,
which not only satisfies the ASR but also allows us to generalize the above
decomposition for the non-locally perturbed surface $\Gamma$ under consideration.

% Briefly introduce the motivation of our new radiation condition

Since $\Gamma$ in (\ref{eq:Gamma}) is no longer a local perturbation of
$\{x_2=0\}$, subtracting $u^{re}$, defined in the last paragraph, from $u^{sc}$
no longer yields an outgoing wave at infinity; clearly, $u^{sc}$ contains a
reflected plane wave $u^{re}_h=-e^{ikx'_h\cdot d}$ with $x_h':=(x_1, -2h-x_2)$,
due to the half line $\{x_2=-h\}$, which is in general different
from $u^{re}$ for $h\neq 0$. It turns out that $u^{re}_h$ and $u^{re}$ reside in
two non-overlapping regions, respectively, which occupy the whole medium above
$\Gamma$ but are separated by a straight line $\ML$ parallel to $d'$, the
propagating direction of $u^{re}$. Consequently, we enforce that subtracting
$u^{re}$ in the region left to $\ML$ and then $u^{re}_h$ in the region right to
$\ML$ from $u^{sc}$, yields an outgoing wave at infinity satisfying (\ref{RC}).
Based on this new radiation condition, we prove uniqueness and existence of weak
solutions by a coupling scheme between finite element and integral equation
methods.

  To numerically compute $u^{sc}$, we place a perfectly matched layer (PML)
  \cite{ber94} to absorb the outgoing waves extracted from $u^{sc}$. Since $u^{sc}$
  differs from outgoing waves by only reflected plane waves propagating parallel to $d'$, we
  impose a homogeneous Robin-type boundary condition to eliminate $u^{re}$ and
  $u^{re}_h$ on the boundary of the PML by following \cite{luluson18}. As the medium
  structure can be decomposed into two $x_1$-uniform segment, the
  numerical mode matching (NMM) method of \cite{luluson18} will be adapted to our scattering problem.
  More precisely, we expand $u^{sc}$ in terms of eigenmodes in each uniform
  segment by separating variables in the governing equations, and then match the
  two mode expansions on the common interface separating the two regions. This
  in turn gives rise to a linear system of the unknown Fourier coefficients in
  the expansions. Solving the linear system yields a numerical solution to our
  scattering problem. Numerical experiments shall be carried out to show the
  performance of our numerical methods based on the new radiation condition.

  The NMM method \cite{chew95,luluson18}, a.k.a
  mode expansion method or modal methods \cite{botcramcp81, li93a, shestesan82}, and its numerous numerical invariants
  \cite{chiyehshi09, gra99, gragui96, kno78, lalmor96, li96, lushilu14, mor95,
  sonyualu11, derdezoly98, biebae01, biederbaeolydez01} are applicable when the structure can be divided into a number of
  segments, where the medium becomes uniform along one spatial variable.
  The classical mode matching method solves the eigenmodes analytically while the
NMM methods solve the eigenmodes by numerical methods,
and they are easier to implement and applicable to more general structures. The
mode matching method and its variants have the advantage of avoiding
discretizing one spatial variable. They are widely used in engineering
applications, since many designed structures are indeed piecewise uniform.

For numerical simulations of waves, the perfectly matched layer (PML)
\cite{ber94,chewee94} is an important technique for truncating unbounded
domains. It is widely used with standard numerical methods, such as finite
element methods \cite{mon03} and spectral methods, etc., that discretize the
whole computational domain. For structures that are piecewise homogeneous, the
boundary integral equation (BIE) methods \cite{cai02, brulyoperaratur16,
  laigreone16} are popular since they can automatically take care of radiation
conditions at infinity while discretizing only interfaces of the structure. For
scattering problems in layered media, PML can also be incorporated with BIE
methods to efficiently truncate interfaces that extend to infinity
\cite{luluqia18}.

The remaining of this paper is organized as follows. In the subsequent section
\ref{sec:well-posedness}, we enforce a new radiation condition on the scattered
field and prove uniqueness and existence of solutions to our globally perturbed
scattering problem. In section 3, we develop a NMM method to solve
  the scattering problem. Section \ref{sec:numerics} is devoted to numerical
examples.

\section{Well-posedness}\label{sec:well-posedness}
This section is devoted to existence and uniqueness of the 2D wave scattering problem in an
upper-half space with a trapezoidal {\color{rot} sound-soft} boundary $\Gamma=\{(x_1,0)|x_1\leq 0\}\cup
\{(0,x_2)|-h\leq x_2\leq 0\}\cup \{(x_1,-h)|x_1\geq 0\}$ for $h>0$. A new radiation condition will be proposed in subsection \ref{subsec1} and our main result, Theorem \ref{Theorem}, will be proved in subsection \ref{subsec2}.

\subsection{Radiation condition}\label{subsec1} Without loss of generality, we
suppose that the incident angle $\theta\in(0, \pi/2)$. This means that the
incoming wave is incident onto $\Gamma$ from the left hand side of the upper
half plane.
We divide the region $\Omega_\Gamma$ into three parts
$\Omega_\Gamma=\Omega^L\cup \Omega^M\cup \Omega^R$ with (see Figure \ref{fig2})
\ben
&&\Omega^L:=\{x\in \Omega_\Gamma: x_2>x_1\tan\theta\},\\
&&\Omega^R:=\{x\in\Omega_\Gamma: x_2<x_1\tan\theta-h/\cos^2\theta\},\\
&&\Omega^M:=\{x\in\Omega_\Gamma:
x_1\tan\theta-h/\cos^2\theta<x_2<x_1\tan\theta\}.\enn

These domains are separated by the following two rays
\ben
&&\ML:=\{x=s(\cos\theta,\sin\theta): s>0\},\\
&& \mathcal{L}':=\{x=s(\cos\theta,\sin\theta)+(0,-h/\cos^2\theta): s>0\}.
\enn

\begin{figure}[!ht]
  \centering
  \includegraphics[width=0.5\textwidth,  viewport=0.pt 130.pt 1000.pt
  750.pt,clip]{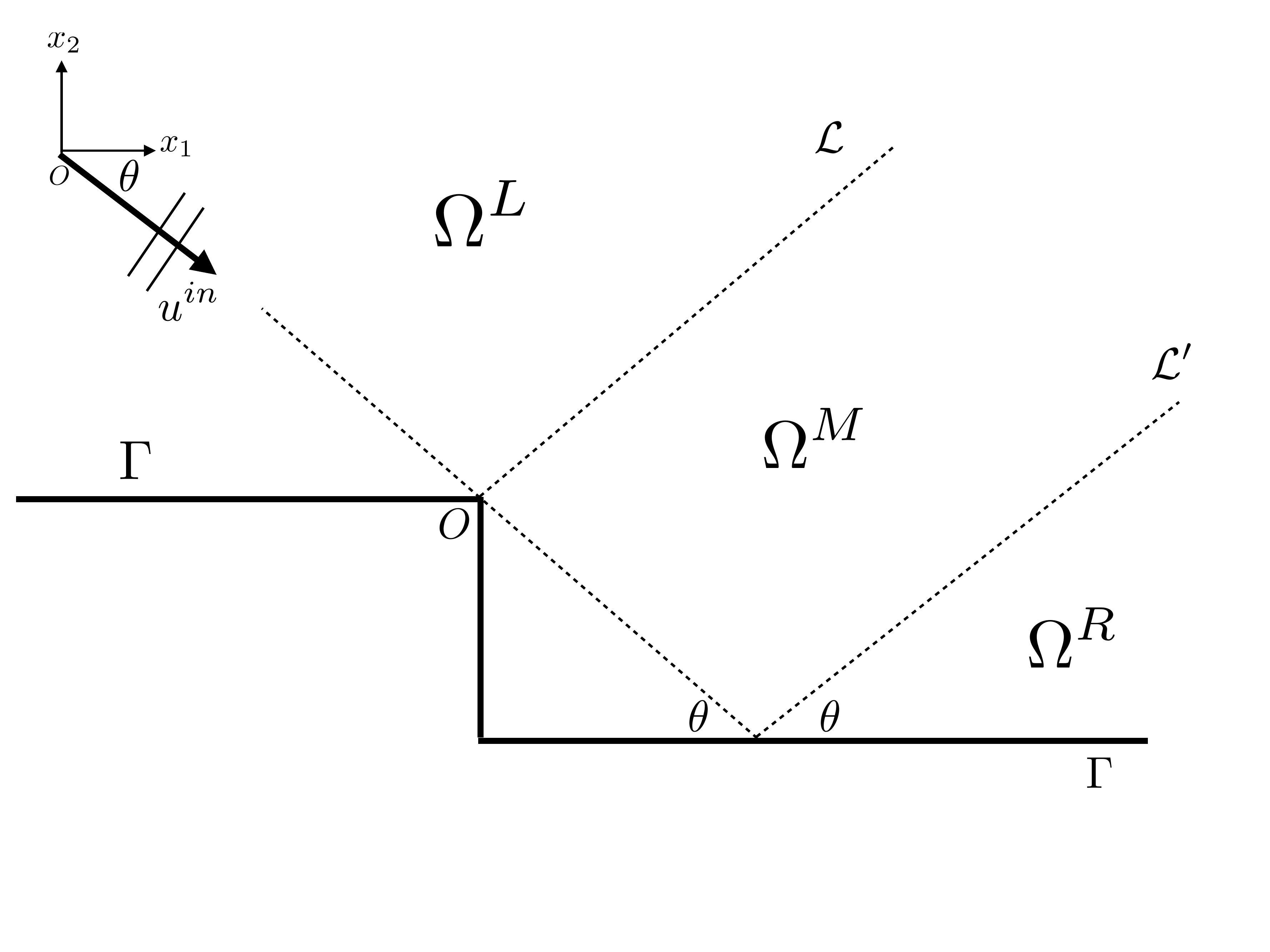}
  \caption{Illustration of the domains $\Omega^L$, $\Omega^R$ and $\Omega^L$
    separated by two rays ${\color{rot}\cal L}$ and ${\color{rot}\cal L}'.$ These domains are determined by the incident angle $\theta\in(0,\pi/2)$ and the height $h$ of $\Gamma$.}
  \label{fig2}
\end{figure}
Denote by $S_r^\eta$ ($\eta=L,M,R$) the restriction of the {\color{rot} circular curve} $S_r$ to $\Omega^\eta$, that is, $S_r^\eta=\{x:  x\in S_r\cap \Omega^\eta \}$.
Let $u^{tot}_L$, $u^{tot}_R$ be the uniquely determined total fields incited by the plane wave $u^{in}$ incident on the sound-soft straight lines $\{x_2=0\}$ and $\{x_2=-h\}$, respectively.  We refer to \cite{S.P, CWEL, ZC98, ZC2003} for uniqueness and existence of the solution in H\"older continuous spaces or in weighted Sobolev spaces using integral equation or variational methods.
Mathematically, they are given explicitly by
\ben
&&u^{tot}_L={\color{rot} u^{in} + u^{re}}=e^{i(\alpha x_1-\beta x_2)}-e^{i (\alpha x_1+\beta x_2) },\\
&&u^{tot}_R={\color{rot} u^{in} + u^{re}_h}=e^{i(\alpha x_1-\beta x_2)}-c_h\,e^{i (\alpha x_1+\beta x_2) },
\enn where $\alpha:=k{\color{rot}\cos\theta}, \beta:=k{\color{rot}\sin\theta}$, $c_h:=e^{2i\beta h}$.
Denote by $u^{tot}$ the total field to our scattering problem.
To introduce our radiation condition, we need to define
 two functions
\be\label{v}
v:=\left\{\begin{array}{lll}
u^{tot}-u^{tot}_L&&\mbox{in}\quad \Omega^L,\\
u^{tot}-u^{tot}_R &&\mbox{in}\quad\Omega^M\cup\Omega^R;
\end{array}\right.\\ \label{v'}
v':=\left\{\begin{array}{lll}
u^{tot}-u^{tot}_L&&\mbox{in}\quad \Omega^L\cup\Omega^M,\\
u^{tot}-u^{tot}_R &&\mbox{in}\quad\Omega^R.
\end{array}\right.
\en
Obviously, $v$ coincides with $v'$ in $\Omega^L\cup\Omega^R$ and $v$ differs from $v'$ only over the region $\Omega^M$.
We remark that, since $u^{tot}\in H^1_{loc}(\Omega_\Gamma)$,  the function $v$ is
discontinuous on ${\color{rot} {\cal L}}$, whereas $v'$ is discontinuous on
${\color{rot} {\cal L}}'$. More precisely, it holds that
\be\label{jump1}\begin{split}
v^+-v^-=(1-c_h) e^{i(\alpha x_1+\beta x_2)}\quad&&\mbox{on}\quad \ML,\\
(v')^+-(v')^-=(1-c_h)e^{i(\alpha x_1+\beta x_2)}\quad&&\mbox{on}\quad \ML'.
\end{split}\en Here, the notation $(\cdot)^\pm$ denote respectively the limits
taking from left and right hand sides. However, the normal derivatives of $v$
(resp. $v'$) are continuous when getting across ${\color{rot} {\cal L}}$ (resp.
${\color{rot} {\cal L}}'$), that is,
\be\label{jump2}\begin{split}
\partial_\nu^+v-\partial_\nu^-v=0\quad&&\mbox{on}\quad \ML,\\
\partial_\nu^+(v')-\partial_\nu^-(v')=0\quad&&\mbox{on}\quad \ML'.
\end{split}
\en
In the following lemma we show that the half-plane Sommerfeld radiation conditions of $v$ and $v'$ are equivalent.
\begin{lemma}\label{lem:radiation} The function
$v$ satisfies the Sommerfeld radiation condition (\ref{RC}) if and only if $v'$ satisfies (\ref{RC}).
\end{lemma}
\begin{proof} We first note that the Sommerfeld radiation condition of $v$  are understood as
\ben
\int_{S_{r}}|\partial_r v-ik v|^2\,ds&=&\left(\int_{S_{r}^L }+\int_{S_r^M\cup S_r^R}\right) |\partial_r v-ik v|^2\,ds\rightarrow
0,\\
\int_{S_{r}}|v|^2\,ds&=&\left(\int_{S_{r}^R }+\int_{S_r^M\cup S_r^L}\right) |v|^2\,ds=O(1),
\enn
respectively, as $r=|x|\rightarrow \infty$. Set $w=v-v'$. Then $w={\color{rot}u^{re}-u_h^{re}}=(c_h-1)e^{i(\alpha x_1+\beta x_2)}$ in $\Omega^M$ and $w=0$ in $\Omega^L\cup\Omega^R$. Hence, we only need to prove that
\be\label{SRCM}
\int_{S_{r}^M}|\partial_r w-ik w|^2\,ds\rightarrow 0,\quad
\int_{S_{r}^M}|w|^2\,ds=O(1)\quad
\mbox{as}\quad r\rightarrow 0.
\en
Let $(r,\varphi)$ be the polar coordinate of $x$, and let $(r,\theta^*(r))$ be
the polar coordinate of the intersection point of $S_r$ and ${\color{rot}\cal L}'$.
Using the fact that
$
w(x)=w(r,\varphi)=(c_h-1) e^{ikr \cos(\theta-\varphi)}
$ together with the mean value theorem,
it is easy to see
\ben
\int_{S_{r}^M}|\partial_r w-ik w|^2\,ds&=&(c_h-1)^2\,r\int_{\theta^*(r)}^{\theta} | ik \cos(\theta-\varphi)-ik|^2\,d\varphi \\
&=&(c_h-1)^2\,r\,k^2\int_{\theta^*(r)}^{\theta} | \cos(\theta-\varphi)-1|^2\,d\varphi \\
&=& O(1/r^2)
\enn
as $r\rightarrow\infty$, since $|\theta^*(r)-\theta|=O(1/r)$.
Analogously, the second relation in (\ref{SRCM}) follows from the inequality
\ben
\int_{S_{r}^M}|w|^2\,ds=(c_h-1)^2\,r\,\int_{\theta^*(r)}^{\theta} \,d\varphi=O(1)
\enn
as $r\rightarrow\infty$.
The proof of the lemma is complete.
\end{proof}

By Lemma \ref{lem:radiation}, our new radiation condition in $\Omega_\Gamma$ is defined as follows.
\begin{definition}\label{def}
The scattered field $u^{tot}-u^{in}$ is said to be outgoing if the function $v$ or $v'$ satisfies the half-plane Sommerfeld radiation condition (\ref{RC}).
\end{definition}

Below we present several remarks concerning this new radiation solution.
\begin{remark}\label{rem:2.3}
\begin{itemize}
\item[(i)] Obviously, the definition of our outgoing radiation condition depends on the incident angle $\theta$ and the height $h$ of the scattering surface $\Gamma$ in the vertical direction. If $h=0$ (i.e., $\Gamma$ is a local perturbation of the original straight line $\{x_1=0\}$), then we see $u_L^{tot}=u_R^{tot}$ and thus the proposed radiation condition could be reduced to the usual condition for scattering problems in a locally-perturbed half-plane {\color{rot}(see e.g., \cite{Bao00, baohuyin18, Jin98, Li2010, Willers1987, Wood2006})}. Moreover, we remark that $u-u^{in}$ still satisfies the Angular Spectrum Representation (\ref{ASP}) for general rough surface problems.
\item[(ii)] In the definition of $v$ (see (\ref{v})), the ray $\ML$  can be replaced by another ray lying in $\Omega_\Gamma$ which starts from any point on $\Gamma$ with the direction $(\cos\theta,\sin\theta)$.
    By the proof of Lemma \ref{lem:radiation},
     the Sommerfeld radiation condition of $v$ is also equivalent to that of the new function defined in the modified domain. Such substitution also applies to $v'$ and the ray $\ML'$.
%The ray $\ML$ (resp. $\ML'$) can be replaced by another ray $\ML_c$ (resp. $\ML'_c$) defined as follows:
%\ben
%\ML_c=\ML-(0,c),\quad \ML'_c=\ML'+(0,c), \quad c>0.
%\enn
%In other words, $\ML_c$ is obtained by shifting $\ML$ along the negative $x_1$-direction, whereas $\ML'_c$ by $\ML'$ along the positive $x_1$-direction. The domains of the functions $v$ and $v'$ should be changed correspondingly.
\end{itemize}
\end{remark}

Suppose that $u-u^{in}$ is an outgoing radiation solution. {\color{rot} By Definition \ref{def},} the total field {\color{rot}$u^{tol}$} can be decomposed into
\be\label{decomposition}
u^{tot}=u_L^{tot}+v\quad\mbox{in}\quad\Omega^L,\qquad
u^{tot}=u_R^{tot}+v\quad\mbox{in}\quad \Omega^M\cup\Omega^R,
\en
where $u_L^{tot}$ and $u_R^{tot}$ are defined in (\ref{utotal}) and the function $v$  fulfills the {\color{rot}Sommerfeld} radiation condition (\ref{RC}).

The main results of this section is stated in the following theorem. Its proof will be carried out in the subsequent section.
\begin{theorem}\label{Theorem}
Assume that the scattering interface $\Gamma$ is given by (\ref{eq:Gamma}) and $u^{in}=e^{ikx\cdot d}$ is a plane wave.
Then the scattering problem (\ref{eq:hlm})-(\ref{eq:bc}) admits a unique solution of the form (\ref{decomposition}). Moreover, it is the unique solution in the
weighted Sobolev space {\color{rot} $H_\varrho^1({\color{rot} \Omega_a })$ for any $a>0$ and $\varrho\in(-1,-1/2)$}.
\end{theorem}

\subsection{\tcr{Proof of Theorem \ref{Theorem}}}\label{subsec2}
Let $G(x,z)=\frac{i}{4}H_0^{(1)}(k|x-y|)$ be the free-space Green's function to the Helmholtz equation, where $H_0^{(1)}$ is the Hankel function of the first kind of order zero. Suppose that
$u^{in}(x,y)=G(x,y)$ is an incoming point source wave emitting from the source position located by $y\in \Omega_\Gamma$.
By \cite{CWEL},
there admits a unique scattered field belonging to the space $H_\varrho^1({\color{rot} \Omega_a })$
for any $a>0$ and $\varrho\in(-1,0)$, which satisfies the {\color{rot} UASR} (\ref{ASP}) in $x_2>a$.
Let $\Phi(x,y)$ ($y\in \Omega_\Gamma$) be the unique total field caused by the incoming point source wave $u^{in}(x,y)$.  Obviously, $\Phi(x,y)$ can be regarded as the Green's function to our scattering problem.
The proof of Theorem \ref{Theorem} relies essentially on the following proposition.
\begin{proposition}\label{Prop}
The Green's function $\Phi(x,y) (y\in \Omega_\Gamma)$ fulfills the half-plane Sommerfeld radiation condition (\ref{RC}).
\end{proposition}

\tcr{When $\Gamma$ is the graph of a $C^{1,1}$-smooth function,
the assertion of Proposition \ref{Prop} for a compactly supported source term is already contained
in \cite[Theorem 5.1]{CRZ1998} but without a detailed proof.
It was further proved in
\cite[Lemma 2.2]{Hu2018} for general rough surface scattering problems that  $\Phi(\cdot,y)\in H_\varrho^1({\color{rot} \Omega_a }\cap \{x\in \R^2:|x_1|>R\})$  for any $a>\max_{x\in \Gamma}\{x_2\}$, $|\varrho|<1$ and $R>|y_1|$, and that
\ben
\lim_{r\rightarrow\infty}\int_{S_{r}\cap\{x_2\geq a\}}|\partial_r \Phi-ik \Phi|^2\,ds\rightarrow0,\quad
\sup_{r>0}
\int_{S_r\cap\{x_2\geq a\}}|\Phi|^2\,ds<\infty, \quad r=|x|.
\enn
Moreover, the above radiation conditions are proved to be equivalent to the classical Sommerfeld radiation condition that
\[
\sqrt{r}\; (\partial_r \Phi-ik\Phi)\rightarrow 0\quad\mbox{as}\quad r\rightarrow\infty,\quad x_2\geq a.
\]
In our case of the non-locally perturbed surface $\Gamma$ given by (\ref{eq:Gamma}), we may choose $a\geq 0$ {\color{rot} because the solution is continuous up to the surface \{$x_2=0$\}.}
In the remaining part of this subsection we shall verify that the total field of our scattering problem for plane wave incidence can be decomposed into the form of (\ref{decomposition}). We will follow the coupling scheme of \cite{baohuyin18, Hsiao, Li2010}
between the finite element and boundary integral equation methods, but
modified to be applicable to our case in a non-locally perturbed half plane.}

{\color{rot} Set $\Omega_{R}^+=\{x: |x|<R\}\cap \Omega_\Gamma$.} Choose $R>h$
such that $k^2$ is not the eigenvalue of $-\Delta$ in $\Omega_R^+$; see Figure
\ref{Fig4}.
% Introduce the Sobolev spaces on an open arc (see e.g., \cite{Mclean}):
%\[\begin{split}
%&H^{1/2}(S_R):=\{v|_{S_R}: v\in H^{1/2}(\partial\Omega_R^+)\},\\
%&\tilde{H}^{1/2}(S_R):=\{u\in H^{1/2}(\partial \Omega_R^+): \mbox{supp} (u) \subset S_R\}.
%\end{split}\]
%Then we denote by $H^{-1/2}(S_R)$ the dual space of $\tilde{H}^{1/2}(S_R)$, and by $\tilde{H}^{-1/2}(S_R)$ the dual space of $H^{1/2}(S_R)$.
\begin{figure}[!ht]
  \centering
   \includegraphics[width=0.5\textwidth,  viewport=0.pt 175.pt 1000.pt
  800.pt,clip]{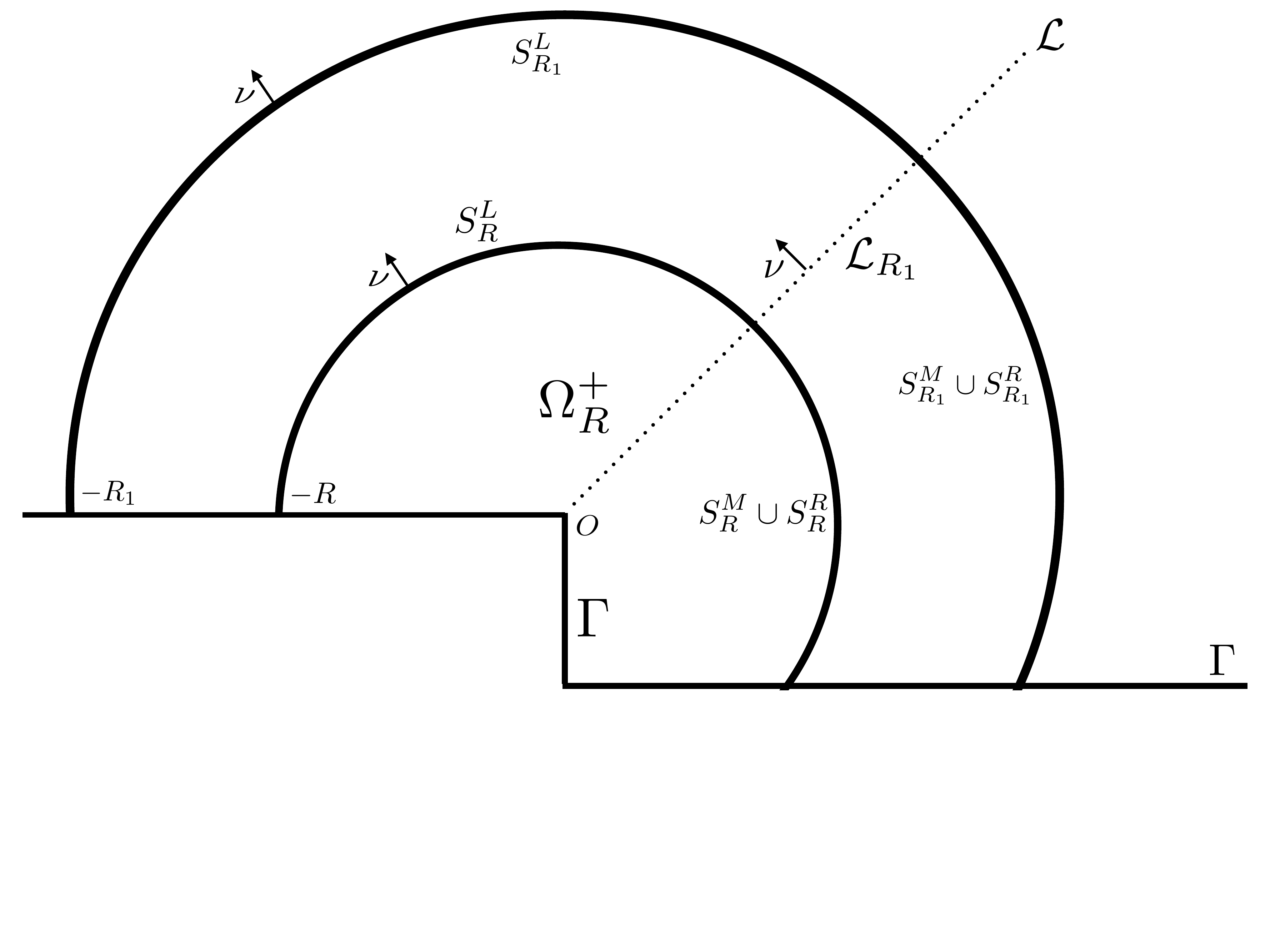}
  \caption{Illustration of the domain $\Omega_R^+$ and the annulus enclosed by $\Gamma$, $S_R$ and $S_{R_1}$. }
  \label{Fig4}
\end{figure}
Choose $R_1>R$. Using Green's formula,
 we may represent $v$ in $\Omega^L\cap\{x: R<|x|<R_1\}$  as the integral representation
\be\nonumber
v(x)&=&\left(\int_{S_R^L}-\int_{S_{R_1}^L}\right)\left[ v(y)\partial_{\nu}\Phi(x;y)-\partial_\nu v(y)\Phi(x;y)\right]\;ds(y)  \\ \label{vL}
&&{\color{rot}+}\int_{\ML_{R_1}} \left[ v^+(y)\partial_{\nu}\Phi(x;y)-\partial_\nu^+ v(y)\Phi(x;y)\right]\;ds(y),
\en with $\ML_{R_1}:=\{x\in \ML: R<|x|<R_1\}$ for $R_1>R$.
Here we have assumed that the normal {\color{rot}vector} $\nu$ on $\ML$ is directed into the left hand side and that on $S_R$ into the exterior $|x|>R$; see Figure \ref{Fig4}. Analogously, for $x\in(\Omega^R\cup\Omega^M)\cap\{x: R<|x|<R_1\}$ we get
\be\nonumber
v(x)&=&\left(\int_{S_R^R\cup S_R^M}-\int_{S_{R_1}^R\cup S_{R_1}^M }\right)\left[ v(y)\partial_{\nu}\Phi(x;y)-\partial_\nu v(y)\Phi(x;y)\right]\;ds(y)  \\ \label{vR}
&&{\color{rot}-}\int_{\ML_{R_1}} \left[ v^-(y)\partial_{\nu}\Phi(x;y)-\partial_\nu^- v(y)\Phi(x;y)\right]\;ds(y)\, .
\en
Adding (\ref{vL}) and (\ref{vR}) together and making use of the jump conditions of $v$ on $\ML$ (see (\ref{jump1}) and (\ref{jump2})), we obtain
\ben
v(x)&=&\left(\int_{S_R}-\int_{S_{R_1}}\right)\left[ v(y)\partial_{\nu}\Phi(x;y)-\partial_\nu v(y)\Phi(x;y)\right]\;ds(y)  \\
&&+{\color{rot}(1-c_h)}\int_{\ML_{R_1}} e^{i(\alpha y_1+\beta y_2)} \partial_{\nu}\Phi(x;y)\;ds(y),
\enn
for $x\in\Omega^L\cap\{x: R<|x|<R_1\}$. \tcr{
In view of the Sommerfeld radiations of $v$ and $G$, we may find that
\ben
&&\int_{S_{R_1}}  v(y)\partial_{\nu}\Phi(x;y)-\partial_\nu v(y)\Phi(x;y)\;ds(y) \\
&=&\int_{S_{R_1}} v(y) \left[\partial_{\nu}\Phi(x;y)-ik\Phi(x;y)\right]-\left[\partial_\nu v(y)-ikv(y)\right]\Phi(x;y)\;ds(y) \\
&\rightarrow& 0,
\enn
as $R_1\rightarrow\infty$.} Hence,
\be\label{vf}
v(x)=\int_{S_R} \left[ v(y)\partial_{\nu}\Phi(x;y)-\partial_\nu v(y)\Phi(x;y)\right]\;ds(y) +f(x)
\en for all $x\in \Omega_\Gamma{\color{rot} \backslash\overline{\Omega_R^+} }$,
where
\be\label{f}
f(x):={\color{rot} (1-c_h) }\lim_{R_1\rightarrow\infty} \int_{\ML_{R_1}} e^{i(\alpha y_1+\beta y_2)} \partial_{\nu}\Phi(x;y)\;ds(y).
\en
The proof of the existence of the limit on the right hand side of (\ref{f}) will be given in the Appendix.
 Note that $f$ vanishes identically if $h=0$ or  $u^{in}=0$, and that
   the integral on $S_R$ appearing in (\ref{vf}) is understood as the sum of those integrals over $S_R^L$ and $S_R^M\cup S_R^R$.
  Since the Green's function $\Phi$ is weakly singular, the jump relation for double layer potentials gives (cf. (\ref{jump1}))
 \[
 f^+-f^-=(1-c_h) e^{i(\alpha x_1+\beta x_2)}\quad \mbox{on}\quad \ML.
 \]
 For notational simplicity, we write $S_1=S_R^L$, $S_2=S_R^M\cup
 S_R^{ {\color{rot} R} }$ and define
 \[
 p_j:=(\partial_\nu v)|_{S_j}\in H^{-1/2}(S_j),\; v_j=v|_{S_j}\in  H^{1/2}(S_j),\,f_j=f|_{S_j}\in  H^{1/2}(S_j),\quad j=1,2.
 \]
 Taking the limit $x\rightarrow S_{R}$, we get
\be\label{Variation-2}
(I-\mathcal{D})\begin{pmatrix}
v_1 \\ v_2
\end{pmatrix}
+\mathcal{S} \begin{pmatrix}
p_1 \\ p_2
\end{pmatrix}=2\begin{pmatrix}
f_1 \\ f_2
\end{pmatrix}\quad\mbox{on}\quad S_1\times S_2.
\en
Here $I$ is the identify operator,  $\mathcal{D}$ and $\mathcal{S}$ are  defined by
\ben
\mathcal{D}=
\begin{pmatrix}
D_{11} & D_{12} \\ D_{21} & D_{22}
\end{pmatrix},\quad
(D_{ij} g)(x)&:=&2\int_{S_j} \partial_{\nu(y)} \Phi(x;y) g(y)\,ds(y) ,\qquad x\in S_i,\\
\mathcal{S}=\begin{pmatrix}
S_{11} & S_{12} \\ S_{21} & S_{22}
\end{pmatrix},\quad
(S_{ij} q)(x)&:=&2\int_{S_j} \Phi(x;y) q(y)\,ds(y),\qquad\qquad x\in S_i,
\enn
for $i,j=1,2$.
We remark that the jump relations for $D_{ij}$ and $S_{ij}$  remain valid,
 since $\Phi(\cdot; y)-G(\cdot; y)$ is of $C^\infty$-smoothness.
On the other hand, using integration by part we may find
\ben
\int_{\Omega_R^+} {\color{rot} \left[\nabla u^{tot}\cdot \nabla\,\overline{\varphi}-k^2u^{tot}\,\overline{\varphi}\right]}\,dx-\int_{S_R}\partial_\nu v\overline{\varphi}\,ds=
\int_{S_R}\partial_\nu (u^{tot}-v)\overline{\varphi}\,ds,
\enn
for all $\varphi\in H^1(\Omega_R^+)$ such that $\varphi=0$ on $\Gamma\cap \{x: |x|<R\}$. This implies that
\be\label{Variation-1}
\int_{\Omega_R^+} {\color{rot}\left[\nabla u^{tot}\cdot \nabla\,\overline{\varphi}-k^2u^{tot}\,\overline{\varphi}\right]}\,dx-\sum_{j=1}^2\int_{S_j}p_j\overline{\varphi}\,ds=\sum_{j=1}^2\int_{S_j} \partial_{\nu} w \overline{\varphi}\,ds\;:=\int_{S_R} \partial_{\nu} w \overline{\varphi}\,ds   ,
\en
where
the function $w$ is defined by (cf. (\ref{v}))
\be\label{w}
w:=u^{tot}-v=\left\{\begin{array}{lll}
u_L^{tot}\quad\mbox{in}\quad \Omega^L,\\
u_R^{tot}\quad\mbox{in}\quad\Omega^R\cup\Omega^M.
\end{array}
\right.
\en
Introduce the variational space $X=X_0\times X_1$, where
\ben
&&X_0=\left\{u\in H^1(\Omega_R^+):\; u=0\;\mbox{on}\;\Gamma\cap\{x: |x|<R\}\right\},\\
&&X_1:=H^{-1/2}(S_1)\times  H^{-1/2}(S_2).
\enn
Combining (\ref{Variation-1}) and (\ref{Variation-2}) gives the variational formulation for the unknown solution pair  $(u^{tot},p)\in X$ with $p=(p_1, p_2)^\top\in X_1$ as follows
\be\label{Variational}
A\left( (u^{tot},p), (\varphi,\chi) \right):=
\begin{pmatrix}
a_1\left( (u^{tot},p), (\varphi,\chi) \right)\\ a_2\left( (u^{tot},p), (\varphi,\chi) \right)
\end{pmatrix}
=\begin{pmatrix}
\int_{S_R}\partial_{\nu}w\,\overline{\varphi} ds\\
2\hat{f}+\int_{S_1\times S_2}(I-\mathcal{D})(\hat{w})\cdot \overline{\chi}\,ds
\end{pmatrix}
\en for all $(\varphi,\chi)\in X$ with $\chi=(\chi_1,\chi_2)^\top$,
where $\hat{f}:=(f|_{S_1}, f|_{S_2})^\top$ and
\ben
&&a_1\left( (u^{tot},p), (\varphi,\chi) \right):=\int_{\Omega_R^+} \nabla u^{tot}\cdot \nabla\,\overline{\varphi}-k^2 u^{tot}\,\overline{\varphi}dx-\sum_{j=1}^2\int_{S_j} p_j\,\overline{\varphi}\, ds, \\
&&a_2\left( (u^{tot},p), (\varphi,\chi) \right):=\int_{S_1\times S_2}\left[ (I-\mathcal{D})(\hat{u}^{tot})+\mathcal{S}p\right]\cdot\overline{\chi}\,ds.
\enn
{\color{rot}Note that $\hat{u}^{tot}:=({u}^{tot}|_{S_1}, {u}^{tot}|_{S_2})^\top$.}
Here we have used the notation
\[
\int_{S_1\times S_2}
\begin{pmatrix}
\xi_1 \\ \xi_2
\end{pmatrix} \cdot
\begin{pmatrix}
\eta_1 \\ \eta_2
\end{pmatrix} ds
:=\begin{pmatrix}
\int_{S_1} \xi_1\eta_1\,ds \\
\int_{S_2} \xi_2\eta_2\,ds
\end{pmatrix}.
\]
%Note that the integrals on $S_R$ appearing in (\ref{Variation-1}) and  (\ref{Variational}) are understood as the sum of those integrals over $S_R^L$ and $S_R^M\cup S_R^R$.

In comparison with the variational formulation for local perturbation scattering problems (see e.g., \cite{Hu2018, Li2010}), we have an additional term $\hat{f}$ appearing on the right hand side of (\ref{Variational}), due to the fact that $h>0$. In addition, the integral over $S_R$ is split into the sum of corresponding integrals over $S_1$ and $S_2$, because of the unknown functions $p_1$ and $p_2$.
 The operator $A$ takes the same form as the case of a locally perturbed half plane. Hence, using mapping properties of $D_{ij}$, $S_{ij}$ and arguing analogously to \cite{Hu2018}, one can prove that
 $A: X\rightarrow X^*$ is a Fredholm operator with index zero. We omit the details, since the proof of \cite{Hu2018, Li2010} carries over to our case easily.

  To prove uniqueness, we assume that $u^{in}=0$. This implies that $w=0$ and $f=0$; {\color{rot} recall (\ref{f}) and (\ref{w}) for the definition}. Hence, the right hand side of the variational formulation (\ref{Variational}) vanishes and $u^{tot}=v$ fulfills the radiation condition (\ref{RC}).
  Since $k^2$ is not the  Dirichlet eigenvalue of $-\Delta$, one can extend the solution {\color{rot}$u^{tol}$} of the homogeneous boundary value problem
$A\left( ({\color{rot}u^{tol}},p), (\varphi,\chi) \right)=0$ for all $(\varphi,\chi)\in X$ from $\Omega_R^+$ to the whole half plane $\Omega_\Gamma$, which also satisfies the Sommerfeld radiation condition (\ref{RC}).
The extended solution also fulfills the Angular Spectrum Representation. Hence, by uniqueness to rough surface scattering problem under the geometrical condition (\ref{shape-assumption}) (see \cite{CWEL}), we get $u^{tot}=0$, which proves the uniqueness of solutions to the problem (\ref{Variational}). Existence follows straightforwardly from Fredholm alternative theory.
This finishes the proof of Theorem \ref{Theorem}. \hfill$\Box$

Once $u^{tot}\in H^1(\Omega_R)$ (and thus $v$) is obtained from (\ref{Variational}), the solution $u^{tot}$ can be extended from $\Omega_R^+$ to the region $|x|>R$ via $u^{tot}=w+v$, where $v$ is expressed by (\ref{vf}) in terms of the trace of $v$ on $S_R$ and the function $f$.

\begin{remark}\label{remark}  Theorem \ref{Theorem} can be readily carried over to the following cases:
\tcr{
\begin{itemize}
\item[(i)] The incident angle $\theta\in (\pi/2,\pi)$, or the scattering
  interface is a local perturbation of $\Gamma$ defined by (\ref{eq:Gamma})
  (that is, the interface coincides with $\Gamma$ in the exterior of a compact
  set). The non-local surface shown in Figure \ref{Fig3} can be analogously
  treated as well. Note that our approach applies to a half-plane which
  satisfies the geometrical assumption (\ref{shape-assumption}). {\color{rot} As
    indicated by Remark \ref{rem:2.3} (ii), we may choose a ray starting from any point on
    $\Gamma$ with the direction $(\cos\theta,\sin\theta)$ as ${\cal L}$ to
    define $\Omega_L$ as the domain on the left of ${\cal L}$.
     The ray ${\cal L}'$ can be chosen
    as any ray
     parallel to and on the right hand side of ${\cal L}$.
    Then $\Omega_R$ is defined as the domain on the right of ${\cal L}'$ and $\Omega_M$
    as the middle domain between ${\cal L}$ and ${\cal L}'$. One easily checks
    that Lemma \ref{lem:radiation} is still valid.}
\item[(ii)] An inhomogeneous medium with compact contrast function is embedded into $\Omega_\Gamma$. In this case, the wave equation for the total field becomes $(\Delta+k^2 q(x)) u^{tot}=0$ in $\Omega_\Gamma$, where $q\in L^\infty(\Omega_\Gamma)$ and $D:=\mbox{Supp}(1-q)$ is a compact set of $\Omega_\Gamma$. The local perturbation is understood as the scattering effect due to the compactly supported inhomogeneity $q$.
\item[(iii)] The incoming wave is a point source wave emitted from some source position located in $\Omega_\Gamma$, that is, $u^{in}(x)=\Phi(x,z)$ for some $z\in \Omega_\Gamma$. Then the total field can be decomposed into the form (\ref{decomposition}) with
    \be\label{utotal}
u^{tot}_L(x)=G(x;z)-G(x;z'),\qquad
u^{tot}_R=G(x;z)-G(x;z'_h),
\en where $z'=(z_1,-z_2)$, $z'_h=(z_1,-2h-z_2)$ for $z=(z_1, z_2)$.
\end{itemize}}
\end{remark}
\begin{figure}[!ht]
  \centering
  \includegraphics[width=0.4\textwidth,  viewport=0.pt 130.pt 1000.pt
  750.pt,clip]{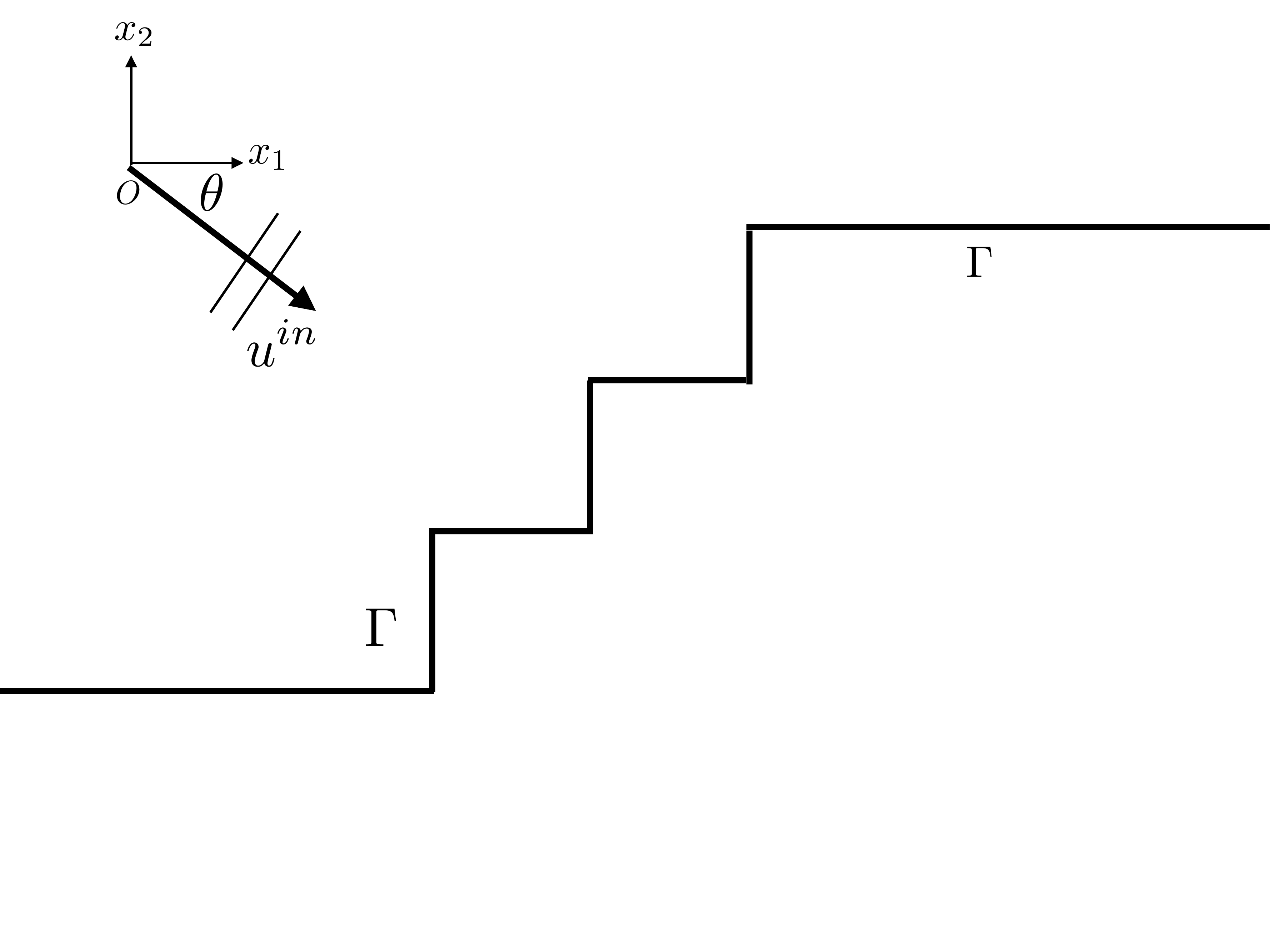}
  \caption{The scattering surface $\Gamma$ is a local perturbation of {\color{rot} the trapezoidal curve} $\{(x_1,0)|x_1\geq 0\}\cup V_h \cup \{(x_1,-h)|x_1\leq 0\}$. The well-posedness result of Theorem \ref{Theorem} extends to this case with a modified radiation condition depending on the angle of the incoming wave. }
  \label{Fig3}
\end{figure}
{\color{rot}\begin{remark}
The fundamental differences between the arguments in \cite{Hu2018} and the proof of Theorem \ref{Theorem} are summarised as follows. In \cite{Hu2018}, a similar coupling scheme was employed to establish well-posedness of time-harmonic acoustic scattering from a locally perturbed sound-soft periodic surface. The mathematical analysis there was mostly placed upon the justification of the Sommerfeld radiation condition of the Green's function (that is, Proposition \ref{Prop}). However, in this paper we consider an acoustic scattering problem in a globally perturbed half plane. Compared to the case of local perturbation, essential difficulties for trapezoidal surfaces arise from the additional term $\hat{f}$ on the right hand side of
the new variational formulation (\ref{Variational}). We have to prove the convergence of the limit in the definition of the function $f$ (see (\ref{f})) , which requires ingenious analysis to estimate the asymptotic behavior of the Green function $\Phi(x,z)$ as $|x|\rightarrow\infty$ uniformly in all $z\in S_R$; see the lengthy arguments in the Appendix for the details.
\end{remark}}

\tcr{The variational formulation established in this section is helpful to establish well-posedness of our scattering problem. However, it is hard to implement the resulting numerical scheme, due to the heavy computational cost on the background Green's function. Instead, a mode matching method will be adopted in the subsequent section to get numerical solutions, where the decomposition form (\ref{decomposition}) will be used to truncate the unbounded domain with an accurate boundary condition.}

\section{A numerical mode matching method}

{\color{rot} Without the knowledge of a precise radiation behavior of the total
  wavefield $u^{tot}$
  infinity, one cannot apply existing truncation
  techniques such as PML method, absorbing boundary condition (ABC) method,
  etc., to truncate $\Omega_\Gamma$ as we don't know at all what boundary
  conditions should be imposed after the truncation, let alone developing
  further numerical methods to compute $u^{tot}$! In this section, we will propose a
  numerical mode matching (NMM) method to compute $u^{tot}$ utilizing the newly
  proposed radiation condition.

  We remark that there are some major differences between the current work and
  \cite{luluson18}. In \cite{luluson18}, an NMM method has been developed for
  the scattering problem in a two-layer medium with a stratified inhomogeneity,
  where an outgoing wavefield is much easier to extract in terms of directly
  subtracting the background reference wavefield from the total wavefield so
  that the aformentioned truncation techniques can be easily applied. Moreover,
  a hybrid Robin-Dirichlet boundary condition was proposed therein on the PML
  boundary to make the mode expansion procedure applicable, but this in fact is
  unnecessary in some standard methods like FEM methods since one can simply put
  zero Dirichlet boundary condition on the whole PML boundary to terminate
  the outgoing wavefield. Unlike \cite{luluson18}, no uniform background
  reference wavefield is available for us to extract an outgoing wavefield for
  the scattering problem under consideration. Whichever one chooses, $u^{tot}_L$
  or $u^{tot}_R$, as the background reference wavefield, (\ref{decomposition})
  tells that neither $u^{tot}-u^{tot}_L$ nor $u^{tot}-u^{tot}_R$ is outgoing
  since the two difference wavefields contain plane waves parallel to $e^{i(
    \alpha x_1 + \beta x_2 )}$ in $\Omega_M\cup\Omega_R$ and $\Omega_L$,
  respectively. Nevertheless, one can use PMLs and then zero Dirichlet boundary
  condition to directly terminate the outgoing wavefield $v$, but one definitely
  meets a ``pseudointerface'' dependent on the incident
  angle, i.e., ray ${\cal L}$, across which $v$ is discontinuous. To avoid such a moving
  pseudointerface, we have two approaches: (1). Directly compute the partially
  outgoing wavefield $u^{tot}-u^{tot}_L$ (or $u^{tot}-u^{tot}_R$) in the whole
  computational domain; we must impose a hybrid Robin-Dirichlet boundary
  condition on the PML boundary to eliminate the reflection of the plane wave
  $e^{i(\alpha x_1 + \beta x_2)}$; (2). Setup a fixed pseudointerface, e.g.,
  $x_1=0$, and then compute $u^{tot}-u^{tot}_L$ on the left of the
  pseudointerface and $u^{tot}-u^{tot}_R$ on the right; like (1), we still need
  to impose a hybrid Robin-Dirichlet boundary condition on the PML boundary
  since one of the two still is partially outgoing. Consequently,
  Robin-Dirichlet boundary condition is necessary in any numerical methods
  unless one could tolerate the movement of pseudointerface as incident angle
  varies. Here, we propose to use the second approach and will develop an NMM
  method to compute $u^{tot}$. }

For simplicity, we assume that $\Gamma$ is defined in (\ref{eq:Gamma}) and there
is no
inhomogeneity above $\Gamma$; this NMM method is applicable as well for
$\Gamma$ locally perturbed with multiple vertical and horizontal segments
and with additional rectangular inhomogeneities above; see
Remark \ref{remark} (i) for details.

%Since $v$ satisfies the half-plane Sommerfeld radiation condition (\ref{RC}), we
%use a perfectly matched layer (PML) \cite{ber94} to truncate $\Omega_\Gamma$
%and expect that the outgoing wavefield $v$ can be aborbed inside the PML.
%Specifically, we introduce the following complex-coordinate transformations
%\[
%  \hat{x}_i = x_i + \int_{0}^{x_i} \sigma_i(t) dt,\quad i=1,2,
%\]
%where the absorbing functions $\sigma_i, i=1,2$ satisfy
%\[
%  \sigma_i(t) = 0,\ {\rm for}\ |t|<L_i, \sigma_i(t)>0,\ {\rm for}\ |t|>=L_i
%\]
%such that $[-L_1,L_1]\times[-L_2,L_2]$ encloses the vertical segment
%$V_h$. Regions with nonzero $\sigma_i$ are called PML
%layers, and we shall compute the total field $u$ only in the truncated region
%$[-L_1-d_1,L_1+d_1]\times[-L_2-d_2,L_2+d_2]\cap \Omega_\Gamma$.
%
%Denoting $\hat{v}(x_1,x_2) = v(\hat{x}_1,\hat{x}_2)$, we see that $\hat{v}$
%satisfies the following PML-transformed Helmholtz equation
%\begin{equation}
%  \label{eq:pmltr:hlm}
%\end{equation}

As  shown in Figure~\ref{fig:trap_pec}, the vertical $x_2$-axis splits $\Omega_{\Gamma}$
into two $x_1$-invariant mediums
\[
\Omega_1 = \Omega_{\Gamma}\cap \{(x_1,x_2)|x_1<0\},\quad{\rm and}\quad \Omega_2 = \Omega_{\Gamma}\cap \{(x_1,x_2)|x_1>0\}.
\]
Introduce the following function
\begin{align}
\label{eq:def:us}
u^{sc}:=\left\{\begin{array}{lll}
u^{tot}-u^{tot}_L&&\mbox{in}\quad \Omega_1,\\
u^{tot}-u^{tot}_R &&\mbox{in}\quad\Omega_2;
\end{array}\right.
\end{align}
 notice that $u^{sc}$ in general is not outgoing.
Then, $u^{sc}$ satisfies
\begin{align}
  \label{eq:hlm:s}
\Delta u^{sc}+k^2u^{sc}&=0\quad\mbox{in}\quad \Omega_1\cup\Omega_2,\\
  \label{eq:bc:s}
u^{sc}&=0\quad\mbox{on}\quad \Gamma/V_h,\\
\quad u^{sc}&= {\color{rot}-}u^{tot}_R\quad\mbox{on}\quad V_h.
\end{align}
In $\Omega_1$, applying the method of separation of variables, we insert
$u^{sc}(x_1,x_2)=\phi(x_2)\psi(x_1)$ into (\ref{eq:hlm:s}) and obtain the
following eigenvalue equations for $\phi$
\begin{align}
  \label{eq:phi:1}
  \frac{d^2\phi}{d x_2^2}  + k^2 \phi &= \beta \phi,\quad{\rm for}\ x_2>0\\
  \label{eq:phi:2}
  \phi(0) &= 0,
\end{align}
and the associated equation for $\psi$
\begin{align}
  \label{eq:psi}
  \frac{d^2\psi}{d x_1^2}   + \beta \psi = 0, \quad{\rm for}\ x_1<0.
\end{align}

Next, along $x_2$-axis, we use a complex-coordinate transformation
\[
  \hat{x}_2 = x_2 + i\int_{0}^{x_2}\sigma_2(t)dt,
\]
where $\sigma_2(t)=0$ for $t<L$ and $\sigma_2(t)>0$ for $t\geq L$ and the half-plane
$x_2\geq L$ with a nonzero absorbing function $\sigma_2$ is called the perfectly
matched layer (PML),  which is capable of absorbing outgoing waves quite efficiently \cite{ber94, chewee94}. According to Eqs. (\ref{v}) and
(\ref{eq:def:us}), $u^{sc}$ and the  outgoing wave $v$ in
$\Omega_1$ differ by a multiple of the reflected plane wave $e^{i\alpha x_1 +
  i\beta x_2}$ in $\Omega_1/\Omega_L$ only when $\Omega_L\subset\Omega_1$. If
$\beta$ is close to $0$, this plane wave can propagate nearly parallel to the
PML entrance $x_2=L$, which numerically causes inefficiency of the PML
absorption. To resolve this issue, our previous work \cite{luluson18} suggests
to eliminate the plane wave by the following Robin-type relation,
\[
  \frac{\partial u^{sc}}{\partial x_2} - i\beta u^{sc} = \frac{\partial v}{\partial x_2} - i\beta v.
\]
The Sommerfeld radiation condition (\ref{RC}) for $v$ implies that $v$ is
approximately an outgoing wave $e^{ik r}$ at infinity so that it decays
rapidly in the PML as $x_2\rightarrow \infty$. Consequently, since
$\partial_{x_2} v -i\beta v$ also produces an outgoing wave at infinity,
\[
  \frac{\partial u^{sc}(\hat{x}_2)}{\partial \hat{x}_2} - i\beta u^{sc} = \frac{\partial v}{\partial \hat{x}_2} - i\beta v
\]
decays as well, implying
\[
  \frac{d\phi(\hat{x}_2)}{d \hat{x}_2}-i\beta \phi(\hat{x}_2),
\]
approaches $0$ for $x_2\rightarrow\infty$. Thus, by setting
$\hat{\phi}(x_2) = \phi(\hat{x}_2)$ and by terminating the PML layer at
$x_2=L+d$ for the PML thickness $d>0$, we get from
(\ref{eq:phi:1}-\ref{eq:phi:2}) that
\begin{align}
  \label{eq:hphi:1}
  \frac{1}{1+i\sigma_2}\frac{d}{d x_2}\left(  \frac{1}{1+i\sigma_2}\frac{d\hat{\phi}}{d x_2}\right)  + k^2 \hat{\phi}= \beta \hat{\phi},\quad{\rm for}\ x_2>0,
\end{align}
with the following boundary conditions
\begin{align}
  \label{eq:hphi:2}
  \hat{\phi}(0) = 0,\\
  \label{eq:hphi:3}
  \hat{\phi}'(L+d) - i(1+i\sigma_2) \beta \hat{\phi}(L+d) \approx 0.
\end{align}
Employing the Chebyshev collocation method in \cite{sonyualu11} to solve
 the above eigenvalue problems Eqs.~(\ref{eq:hphi:1}),
(\ref{eq:hphi:2}), and (\ref{eq:hphi:3}), we obtain $N$ solutions of eigenpairs
$\{\beta_j^1, \hat{\phi}_j^1\}_{j=1}^{N}$ when $N$ collocation points
$\{x_2^j\}_{j=1}^{N}$are used to discretize $x_2\in[0,L+d]$. According to
\cite{luluson18}, ${\rm Im}(\beta_j^1)\geq 0$ so that ${\rm
  Im}(\sqrt{\beta_j^1})\geq 0$ and ${\rm Re}(\sqrt{\beta_j^1})\geq 0$.

Now, inserting each eigenpair into Eq.~(\ref{eq:psi}) yields two independent
solutions $e^{-i\sqrt{\beta_j^1} x_1}$ and $e^{i\sqrt{ \beta_j^1 } x_1}$. We
claim that $u^{sc}$ propagates only towards negative $x_1$-axis  so
  that we choose $\psi_j=e^{-i\sqrt{\beta_j^1} x_1}$ that propagates towards
  negative $x_1$-axis. To show
this, we distinguish two possible cases occurring here. If $\alpha\geq 0$, the
given incident wave $u^{in}$ propagates towards positive $x_1$-axis so that one
easily sees that $\Omega_1\subset\Omega_L$, which indicates that $u^{sc}=v$ is
outgoing in $\Omega_1$. If otherwise $\alpha < 0$, $u^{in}$ now propagates
towards negative $x_1$-axis so that $\Omega_L\subset \Omega_1$, then $u^{sc}$ is
an outgoing wave in $\Omega_1$ plus a multiple of reflected wave $e^{i\alpha x_1
  + i\beta x_2}$ in $\Omega_1/\Omega_L$ which still propagates towards negative
$x_1$-axis. Consequently, $\psi_j=e^{-i\sqrt{\beta_j^1} x_1}$ such that we get
$N$ eigenmodes to approximate $\hat{u}^s$,
\begin{align}
  \label{eq:us:1}
  \hat{u}^{sc}(x_1,x_2) \approx \sum_{j=1}^{N}c_j \hat{\phi}_j^1(x_2) e^{-i\sqrt{\beta_j^1} x_1},\quad{\rm in}\quad \Omega_1,
\end{align}
where $x_2$ is collocated at $\{x_2^j\}_{j=1}^{N}$.

Repeating the same procedure of variable separation in $\Omega_2$, one obtains
$N+M$ eigenmodes to approximate $\hat{u}^s$,
\begin{align}
  \label{eq:us:2}
  \hat{u}^{sc}(x_1,x_2) \approx \sum_{j=1}^{N+M}d_j \hat{\phi}_j^2(x_2) e^{i\sqrt{\beta_j^2} x_1},\quad{\rm in}\quad \Omega_2,
\end{align}
where $x_2$ is collocated at the common points $\{x_2^j\}_{j=1}^{N}$ in
$[0,L+d]$ and also at $M$ extra points $ \{x_2^j\}_{j=N+1}^{N+M}$ in $V_h=[-h,0]$. On $x_1=0$ separating $\Omega_1$ and $\Omega_2$, we have for
$j=1,\ldots, N$ that
\begin{align}
  \label{eq:gov:1}
  u^{sc}(0-,x_2^j) - u^{sc}(0+,x_2^j) &= u_R^{tot}(0,x_2^j) - u_L^{tot}(0,x_2^j),\\
  \label{eq:gov:2}
  \partial_{x_1}u^{sc}(0-,x_2^j) - \partial_{x_1}u^{sc}(0+,x_2^j) &= \partial_{x_1}u_R^{tot}(0,x_2^j) - \partial_{x_1}u_L^{tot}(0,x_2^j),
\end{align}
by (\ref{eq:def:us}), and for $j=N+1,\ldots,N+M$ that
\begin{align}
  \label{eq:gov:3}
  u^{sc}(0+,x_2^j) = u_R^{tot}(0,x_2^j),
\end{align}
by (\ref{eq:bc:s}).

Eqs.~(\ref{eq:gov:1}-\ref{eq:gov:3}) together with the
expansions (\ref{eq:us:1}) and (\ref{eq:us:2}) give rise to a linear system of
$2N+M$ equations for the unknowns $\{c_j\}_{j=1}^N$ and $\{d_j\}_{j=1}^{N+M}$.
Solving this linear system, we get $c_j$ and $d_j$ so that $u^{sc}$ in $\Omega_1$
and $\Omega_2$ are obtained by (\ref{eq:us:1}) and (\ref{eq:us:2}), which
eventually gives the total field $u^{tot}$ by (\ref{eq:def:us}).

As for incident cylindrical waves, the total wave field $\Phi$, post-subtracting
the free-space Green's function $G$, satisfies the Sommerfeld radiation
condition (\ref{RC}) in the whole domain above $\Gamma$. Therefore, the mode
matching procedure described above can be adapted with ease to solve the
scattering problem for cylindrical incident waves; we omit the details here.
\section{Numerical examples}\label{sec:numerics}
In this section, we will carry out several numerical experiments to validate our
newly proposed radiation condition. In all examples, we assume that the
free-space wavelength $\lambda=1$ so that the free-space wavenumber $k_0=2\pi$,
and the refractive index of the background medium above $\Gamma$ is $n=1$.
In setting up the PML, we choose
\begin{equation}
  \label{eq:sigma}
  \sigma_2(t) = (x_2-L)\sigma/d,\quad t\geq L.
\end{equation}
for a positive constant $\sigma$.

\noindent{\bf Example 1.} In the first example, we directly analyze the
scattering problem for $\Gamma$ in (\ref{eq:Gamma}) with $h=\lambda$. We
consider two different incident waves: (1) a plane incident wave with the
incident angle $\theta=\frac{\pi}{6}$; (2) a cylindrical incident wave excited
by the source point $(0.2,0.2)$. We observe the total wave field $u^{tot}$ in the
domain $[-2.5,2.5]\times [-2.5,2.5]$ above $\Gamma$ so that we take $L=2.5$ in the PML.

To validate our numerical solutions, we use $\sigma=70$ and $d=1$ in the PML,
and compute $N=280$ eigenmodes in $\Omega_1$, and $N+M=420$ eigenmodes in
$\Omega_2$, i.e., $M=140$ points are used to discretize $V_h$, in the NMM method,
to get a reference solution $u_{\rm ref}$ for each of the two incident waves, as
shown in Figure~\ref{fig:ex1:1} below.
\begin{figure}[!ht]
  \centering
  a)\includegraphics[width=0.4\textwidth]{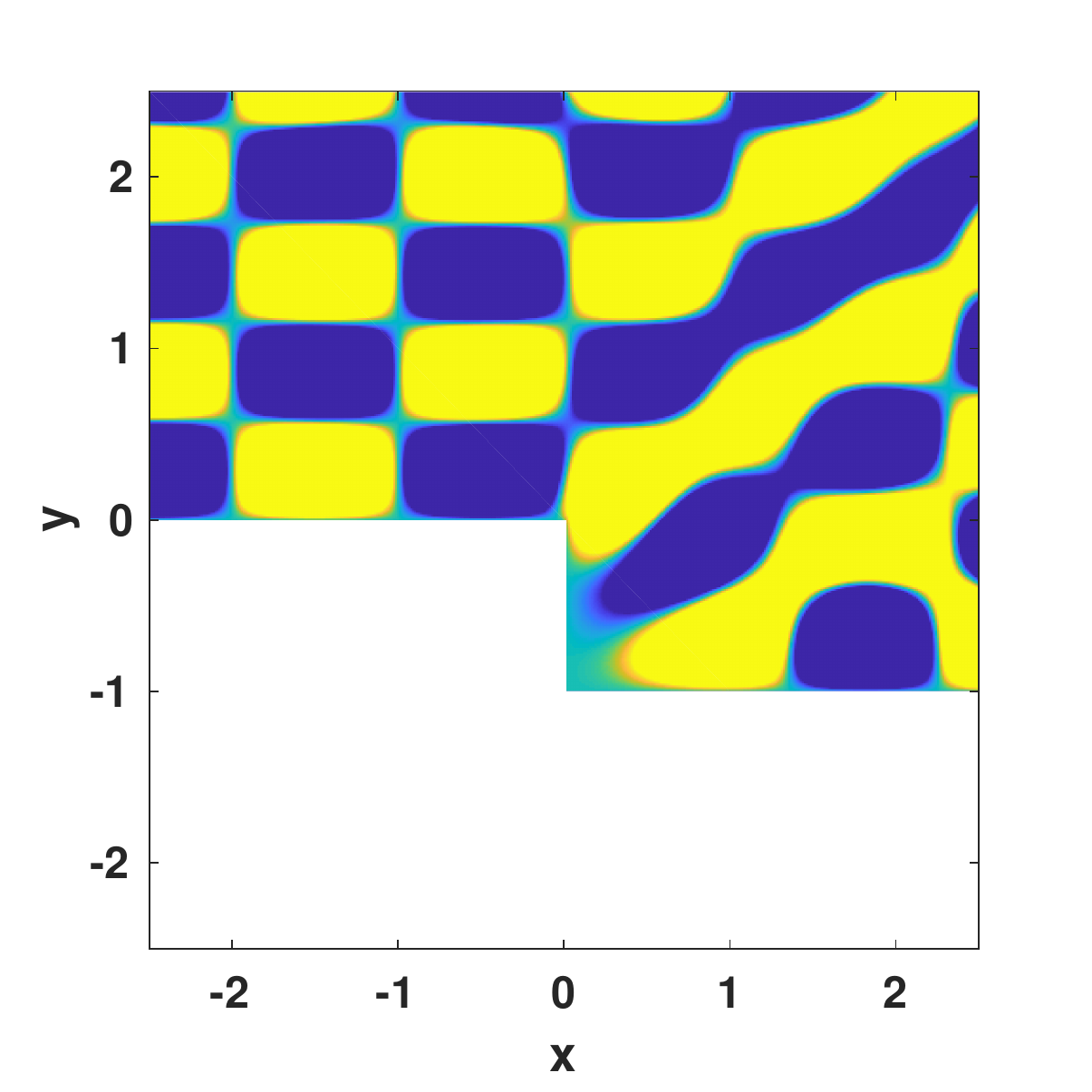}
  b)\includegraphics[width=0.4\textwidth]{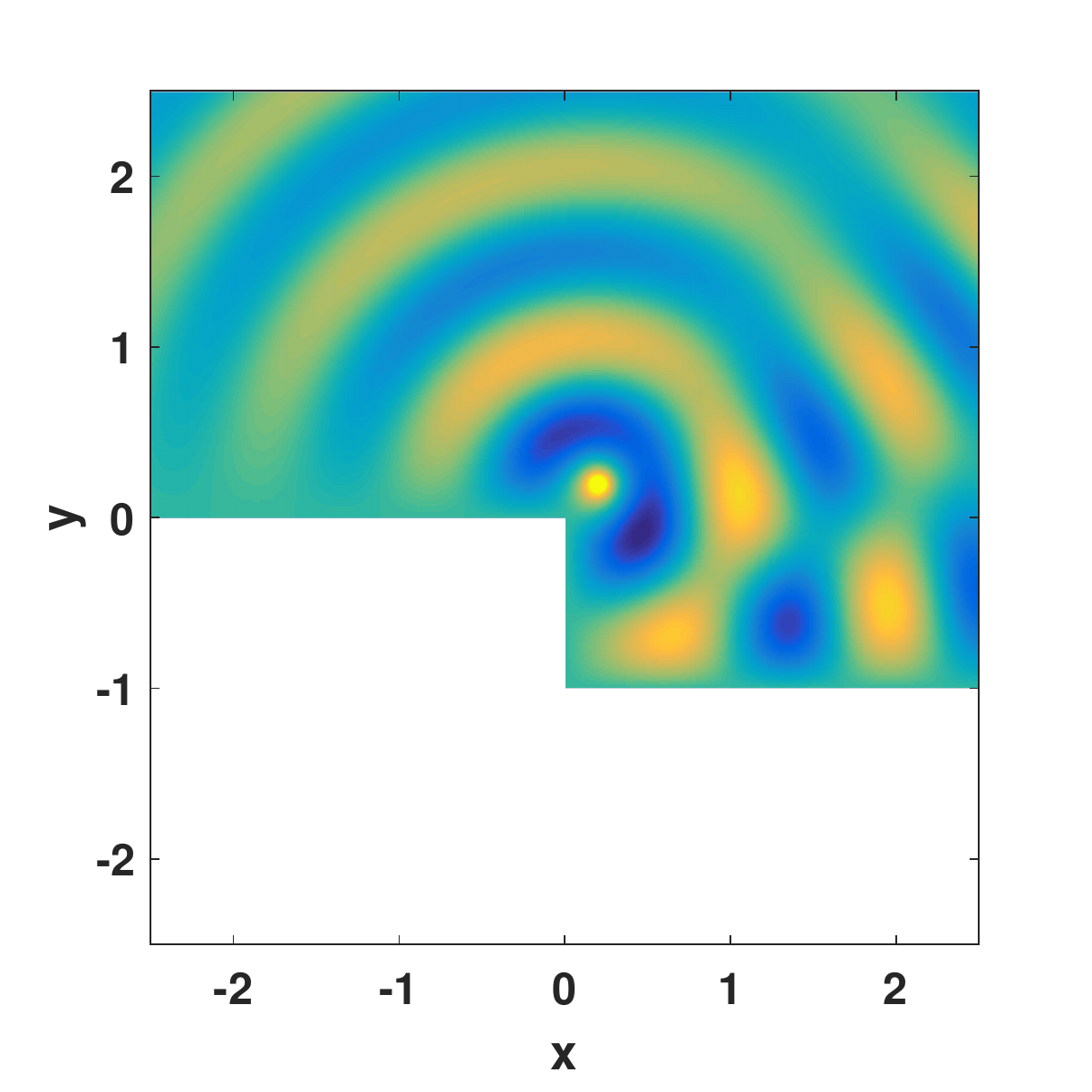}
	\caption{Real part of the total wave field $u^{tot}$ in $[-2.5,2.5]\times[-2.5,2.5]$ for: (a) incident plane wave
    with incident angle $\theta=\frac{\pi}{6}$; (b) incident cylindrical wave
    excited by the source point $(0.2,0.2)$. White region indicates the PEC substrate.  }
  \label{fig:ex1:1}
\end{figure}

To illustrate the absorption efficiency of our PML and to validate the newly
proposed radiation condition (\ref{RC}), we compute the following relative
error,
\begin{align}
  \label{eq:rel:err}
  { E_{\rm rel}} = \frac{\max_{(x,y)\in S}|u^{\rm tot}_{\rm ref}(x,y) - u^{\rm
      tot}_{\rm NMM}(x,y)|}{\max_{(x,y)\in S}|u^{\rm tot}_{\rm ref}(x,y)|},
\end{align}
for different values of $\sigma$ and $d$ in (\ref{eq:sigma}). The set $S=\{(x,y)|x = 0, y=-1, 0, 2.5\}$
defines where numerical solutions and the reference solution are compared; this
choice is typical since it contains all corners of $\Gamma$ and the interior
boundary point of the PML.

Figure~\ref{fig:ex1:2}(a) shows the convergence curve of $E_{\rm rel}$ for
$\sigma=70$ and for different values of $d$, ranging from $0.001$ to $1$; for
a fixed $d=1$, Figure~\ref{fig:ex1:2}(b) shows the convergence curve of
$E_{\rm rel}$ for different values of $\sigma$, ranging from $0.1$ to $70$. We
observe from Figure~\ref{fig:ex1:2} that $E_{\rm ref}$ decays exponentially with the PML parameters
$\sigma$ and $d$ at the beginning when numerical discretization error is not
dominant.

\begin{figure}[!ht]
  \centering
  (a)\includegraphics[width=0.4\textwidth]{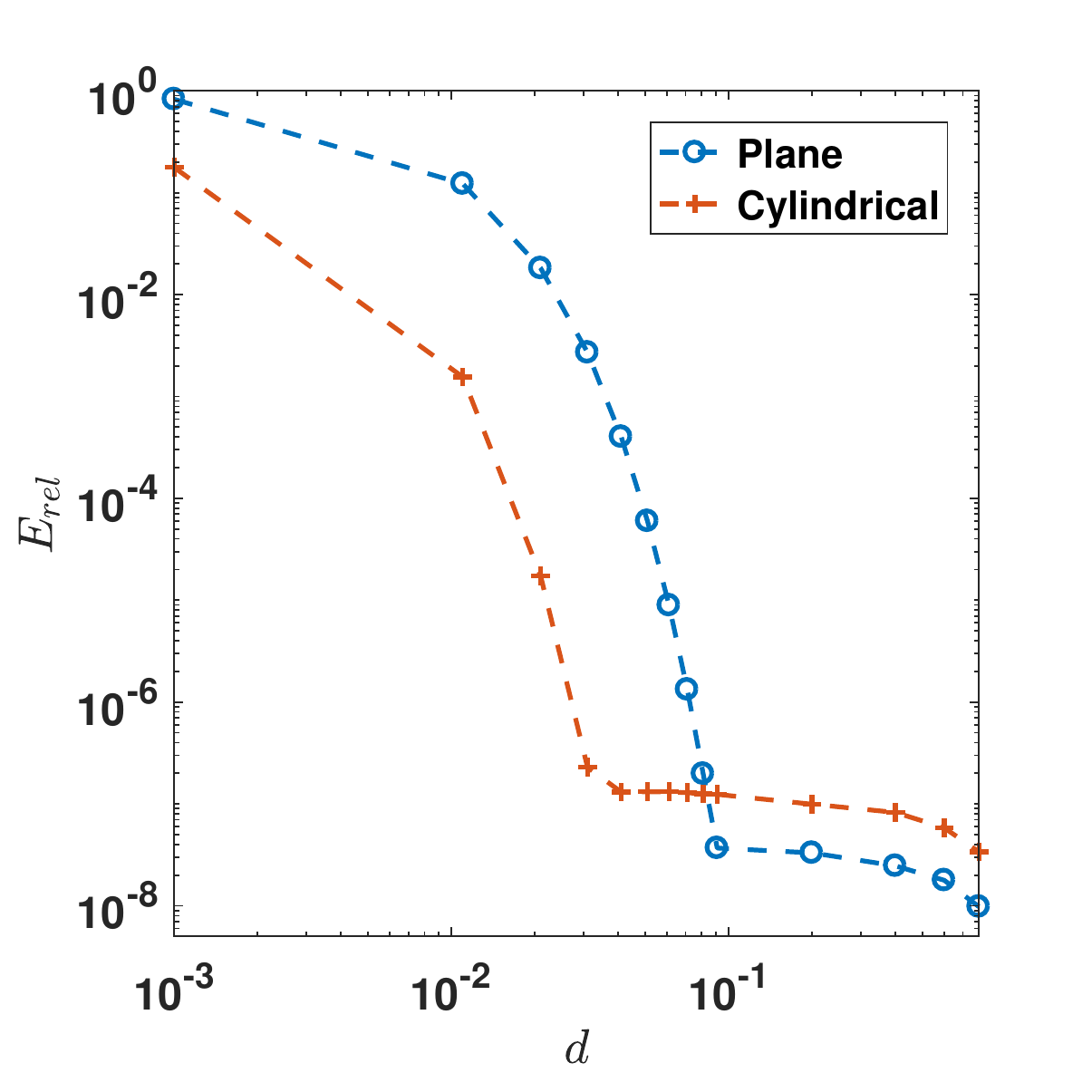}
  (b)\includegraphics[width=0.4\textwidth]{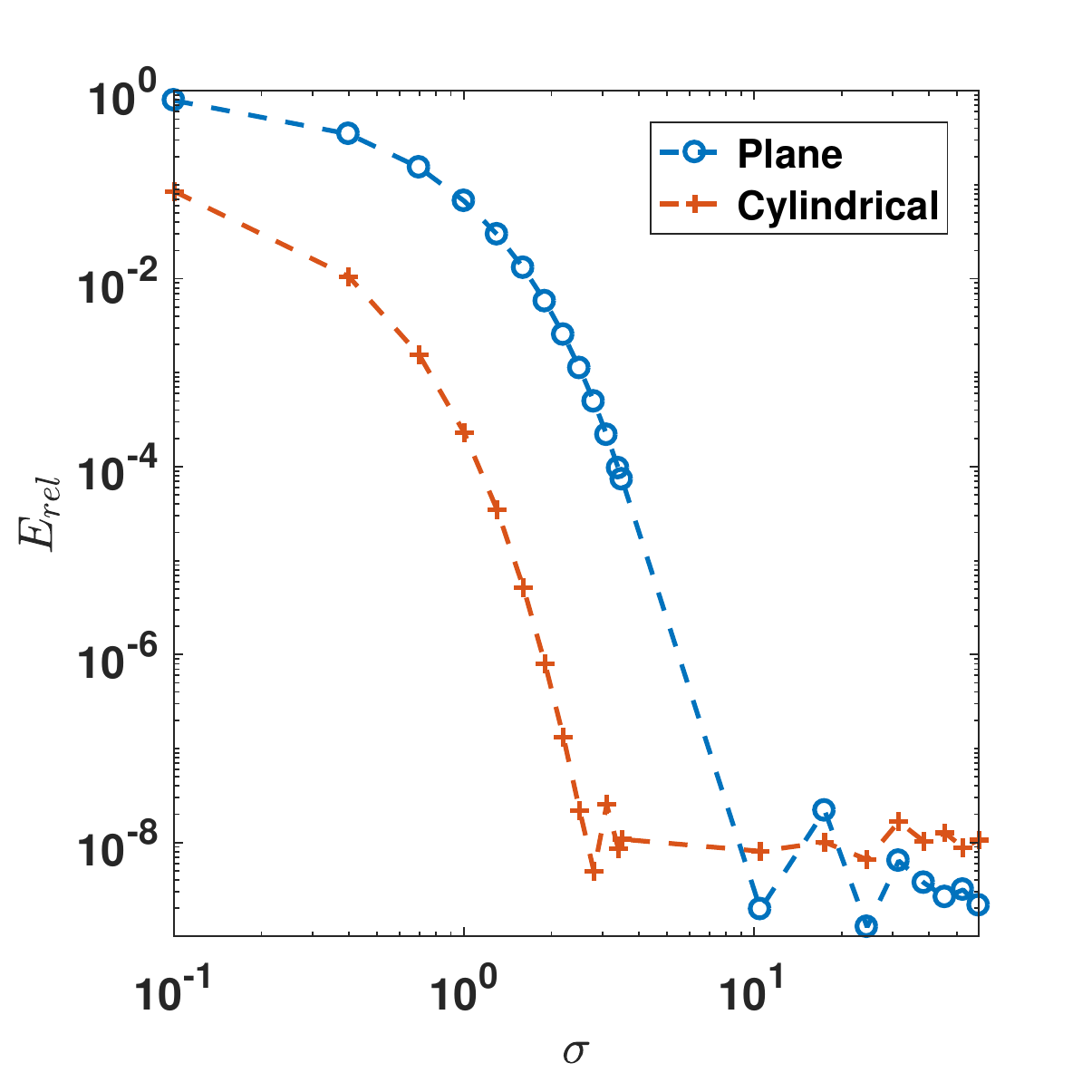}
	\caption{(a) Convergence curve of $E_{\rm rel}$ versus PML thickness $d$
    ranging from $0.001$ to $1$ when $\sigma=70$; (b) Convergence curve of
    $E_{\rm rel}$ versus $\sigma$ ranging from $0.1$ to $70$ when $d=1$.
    lines marked with 'o' indicate curves for plane incident waves, while lines
    marked with '+' indicate curves for cylindrical incident waves.}
  \label{fig:ex1:2}
\end{figure}

\noindent{\bf Example 2.} In this example, we slightly modify the medium in
Example 1 by attaching a penetrable medium of refractive index $n=2$ to the
vertical segment $V_h$ of $\Gamma$, as shown in Figure~\ref{fig:ex2:1}(c).
Again, we compute total wavefields in $[-2.5,2.5]\times[-2.5,2.5]$ for the same
two incident waves used in Example 1. To apply the NMM method, we now need to
split the medium above the PEC surface into three $x$-uniform regions
$\Omega_1$, $\Omega_2$ and $\Omega_3$ separated by $x=0$ and $x=1$. We use $280$
eigenmodes to express the wavefield in $\Omega_1$, and $420$ eigenmodes in the
other two regions $\Omega_2$ and $\Omega_3$. Again, reference solutions are
obtained by setting $\sigma=70$ and $d=1$ in the PML, as shown in
Figure~\ref{fig:ex2:1} (a) and (b).
\begin{figure}[!ht]
  \centering
  a)\includegraphics[width=0.277\textwidth]{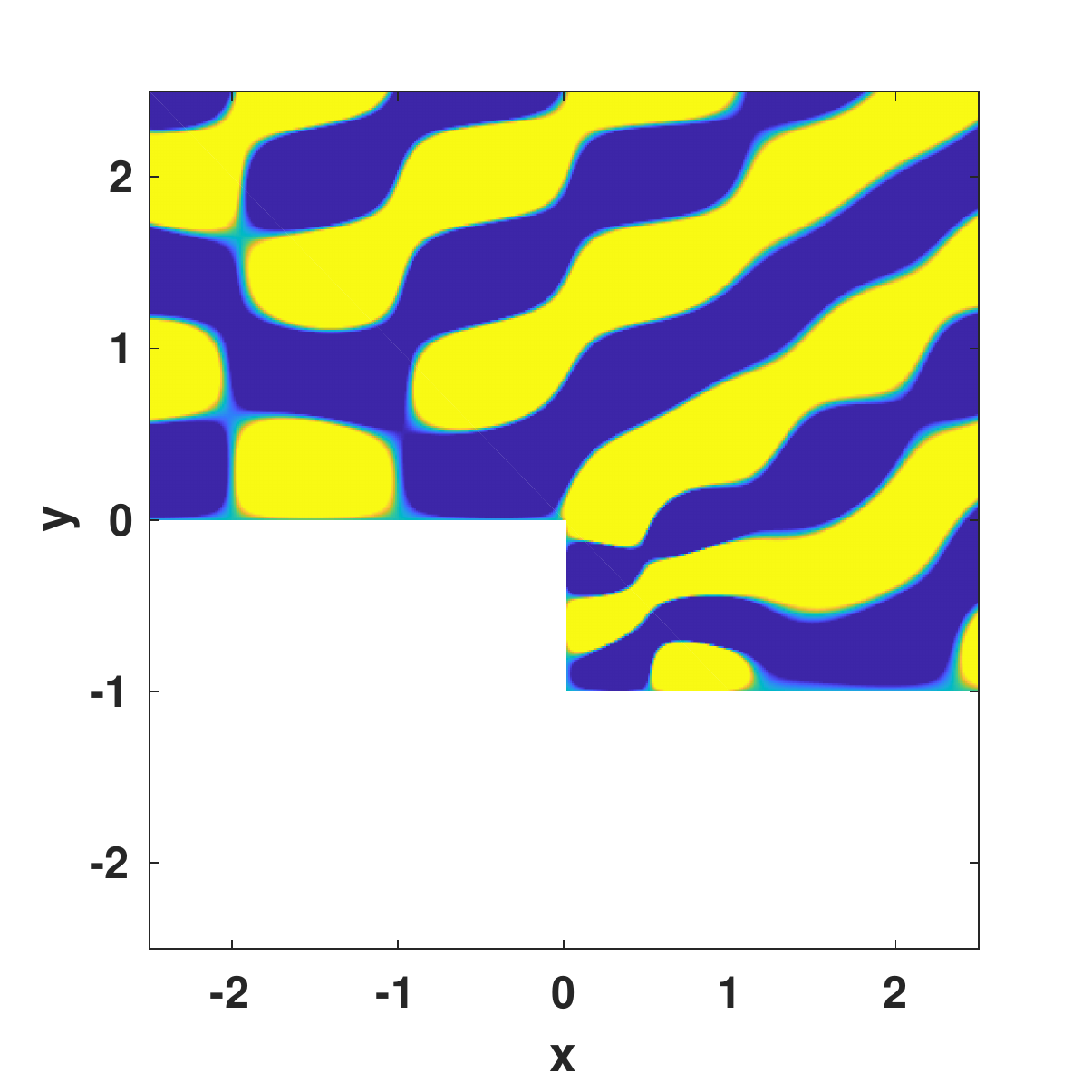}
  b)\includegraphics[width=0.277\textwidth]{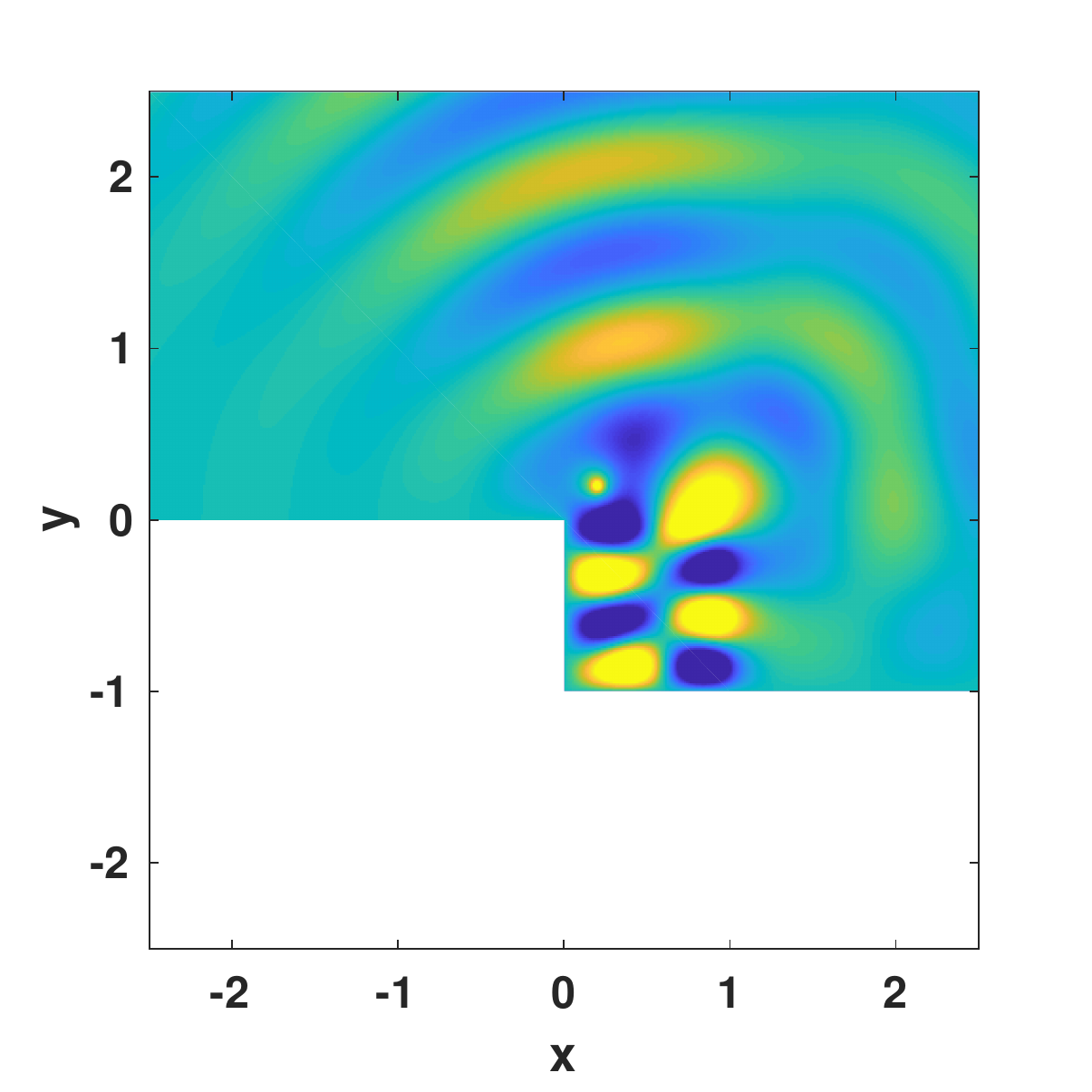}
  c)\includegraphics[width=0.327\textwidth]{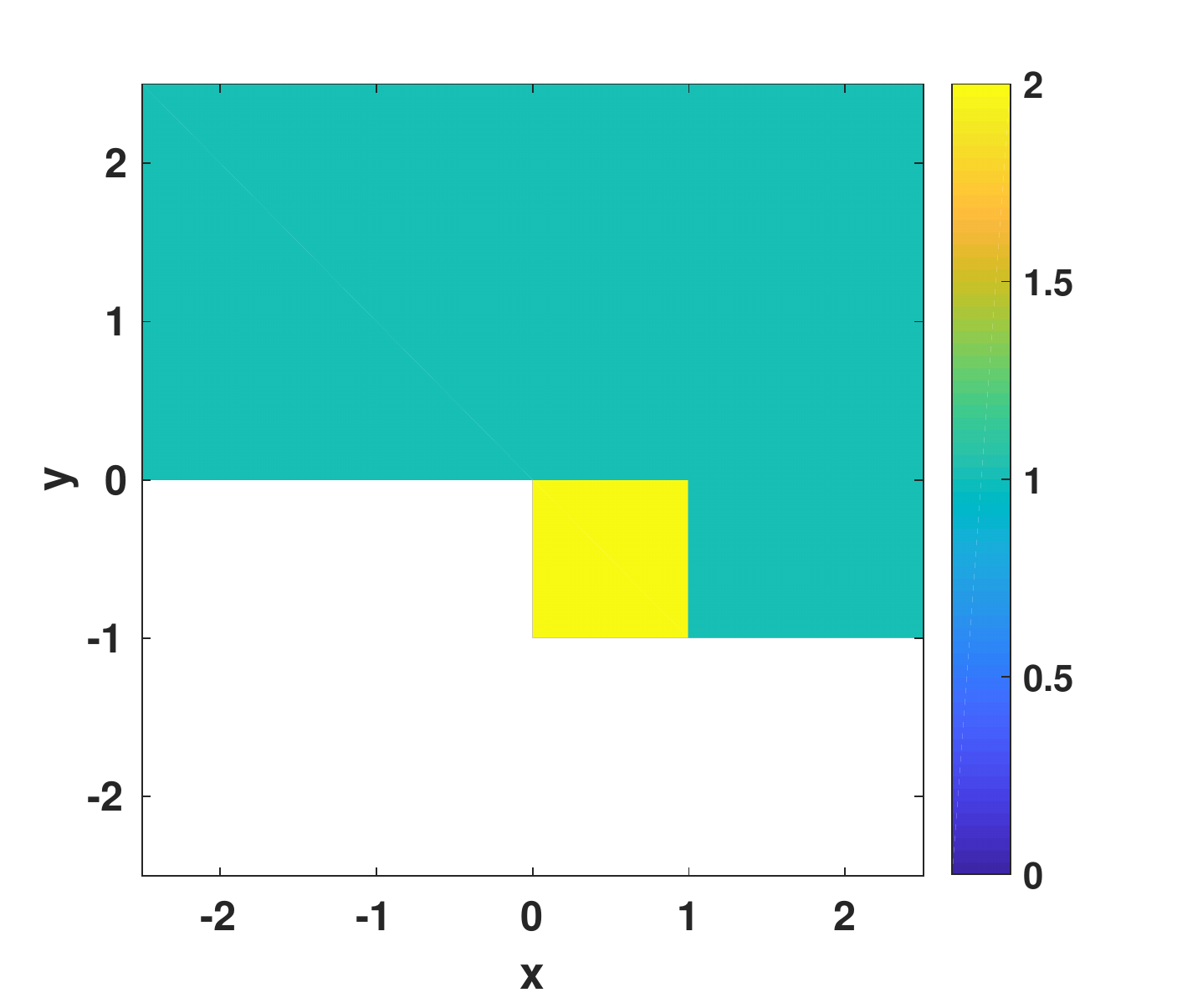}
	\caption{Real part of the total wave field $u^{tot}$ in $[-2.5,2.5]\times[-2.5,2.5]$
    for: (a) incident plane wave with incident angle $\theta=\frac{\pi}{6}$; (b)
    incident cylindrical wave excited by the source point $(0.2,0.2)$. White
    region indicates the PEC substrate. (c) Profile of the refractive index
    above the PEC substrate indicated by the white region. }
  \label{fig:ex2:1}
\end{figure}

To validate the absorption efficiency of our PML, we compute $E_{\rm rel}$
defined in (\ref{eq:rel:err}) for different choices of $\sigma$ and $d$, where
now $S = \{(x,y)|x=0,1,y=-1,0,2.5\}$ is chosen to include all corner points.

Figure~\ref{fig:ex2:2}(a) shows the convergence curve of $E_{\rm rel}$ for
$\sigma=70$ and for different values of $d$, ranging from $0.001$ to $1$; for
a fixed $d=1$, Figure~\ref{fig:ex2:2}(b) shows the convergence curve of
$E_{\rm rel}$ for different values of $\sigma$, ranging from $0.1$ to $70$. As
in Example 1, we observe from Figure~\ref{fig:ex2:2} that $E_{\rm ref}$ again
decays exponentially with the PML parameters $\sigma$ and $d$ at the beginning
when numerical discretization error is not dominant.
\begin{figure}[!ht]
  \centering
  (a)\includegraphics[width=0.4\textwidth]{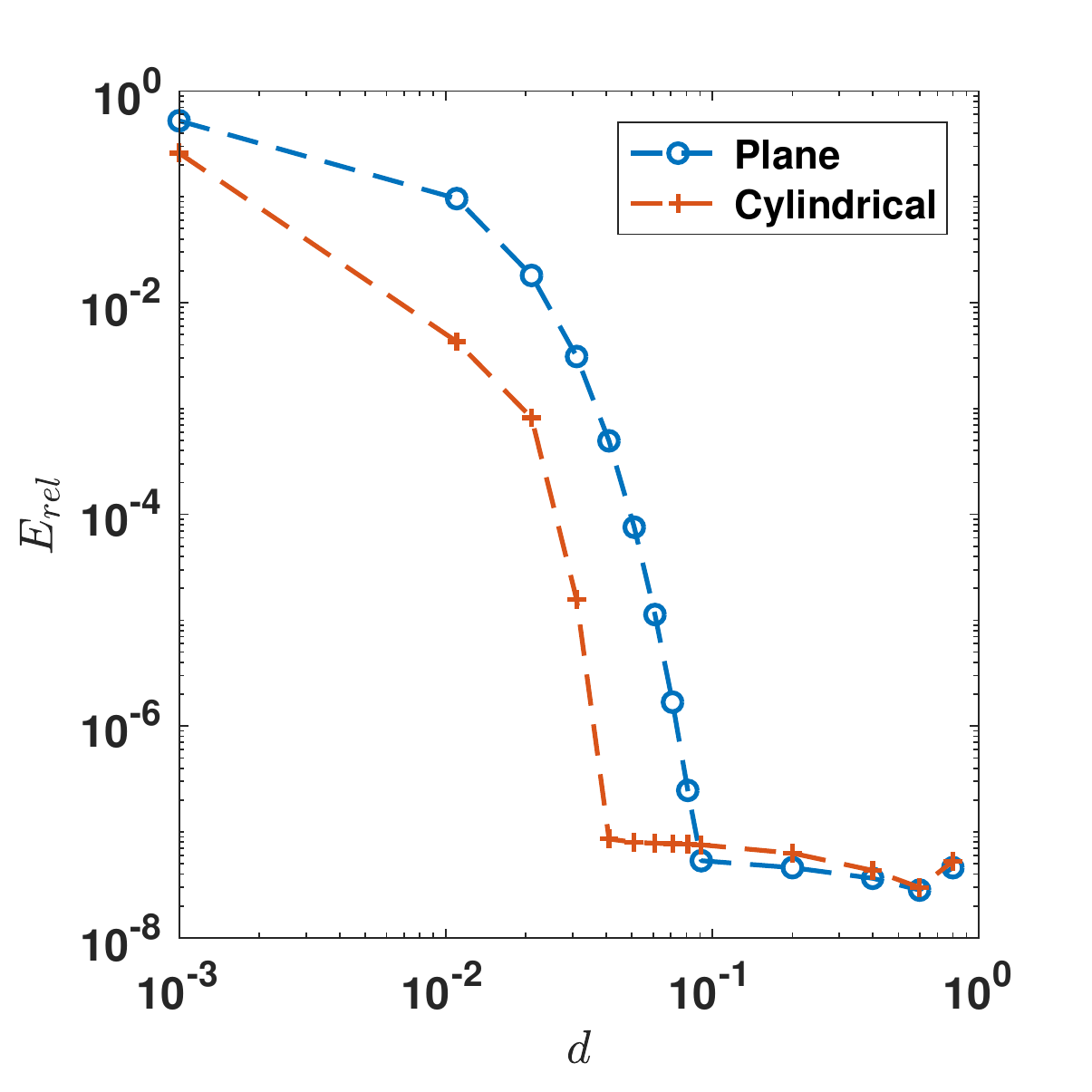}
  (b)\includegraphics[width=0.4\textwidth]{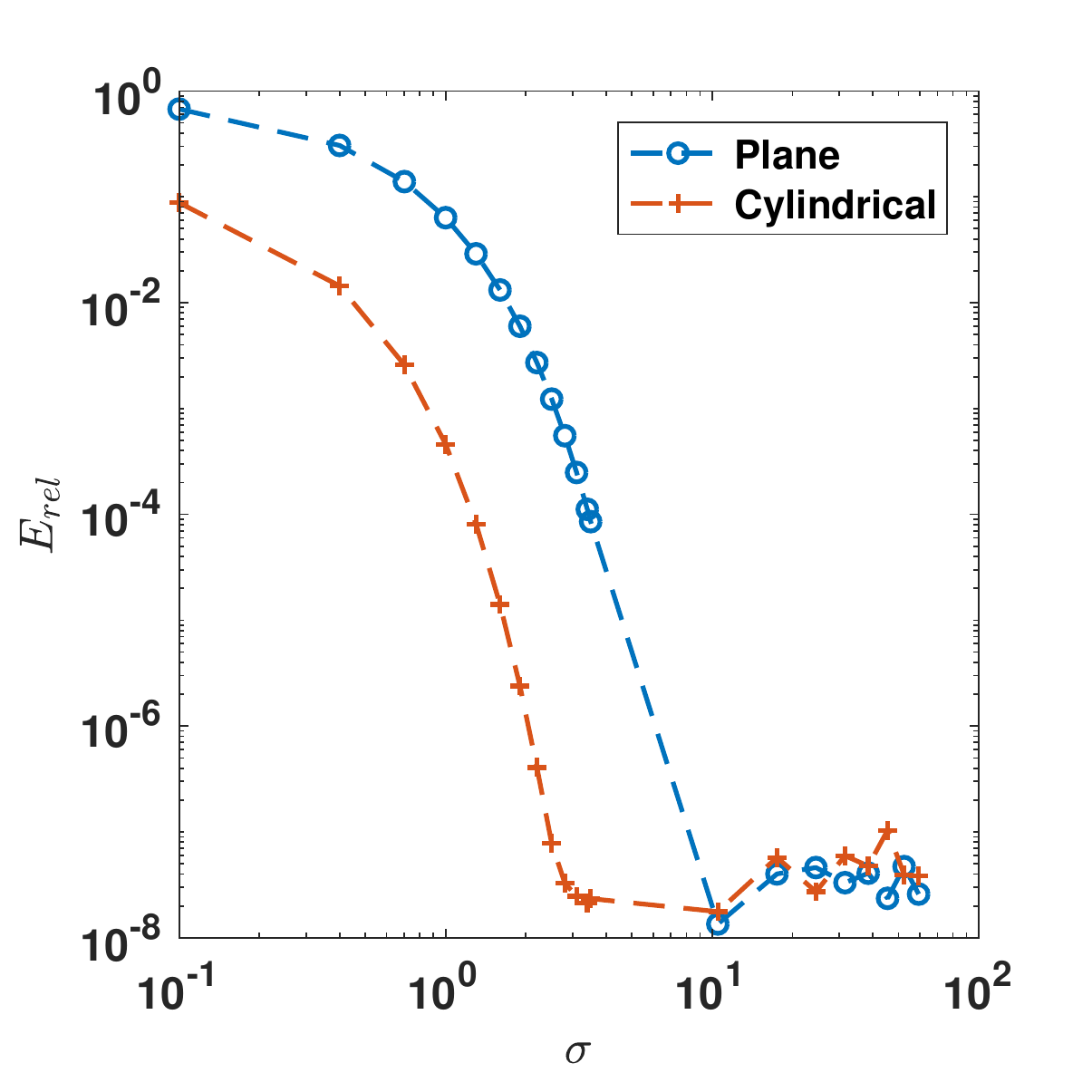}
	\caption{(a) Convergence curve of $E_{\rm rel}$ versus PML thickness $d$
    ranging from $0.001$ to $1$ when $\sigma=70$; (b) Convergence curve of
    $E_{\rm rel}$ versus $\sigma$ ranging from $0.1$ to $70$ when $d=1$.
    lines marked with 'o' indicate curves for plane incident waves, while lines
    marked with '+' indicate curves for cylindrical incident waves.}
  \label{fig:ex2:2}
\end{figure}

\noindent{\bf Example 3.} In the last example, unlike the previous two examples,
we make several indentations and put one penetrable medium on the PEC substrate,
as shown in Figure~\ref{fig:ex3:1}(c). Two incident waves are considered here:
(1) a plane incident wave with incident angle $\theta=\frac{\pi}{6}$; (2) a
cylindrical incident wave excited at the source point $(7.2,1.2)$. With totally
3640 eigenmodes to express wavefield in all $x$-uniform regions, we compute two
numerical solutions for the two incident waves respectively in $[-2.5,10.5]\times[-2.5,2.5]$, where we take
$\sigma=70$ and $d=1$ in the PML. Numerical results are shown in
Figure~\ref{fig:ex3:1}(a) and (b).
\begin{figure}[!ht]
  \centering
  a)\includegraphics[width=0.277\textwidth]{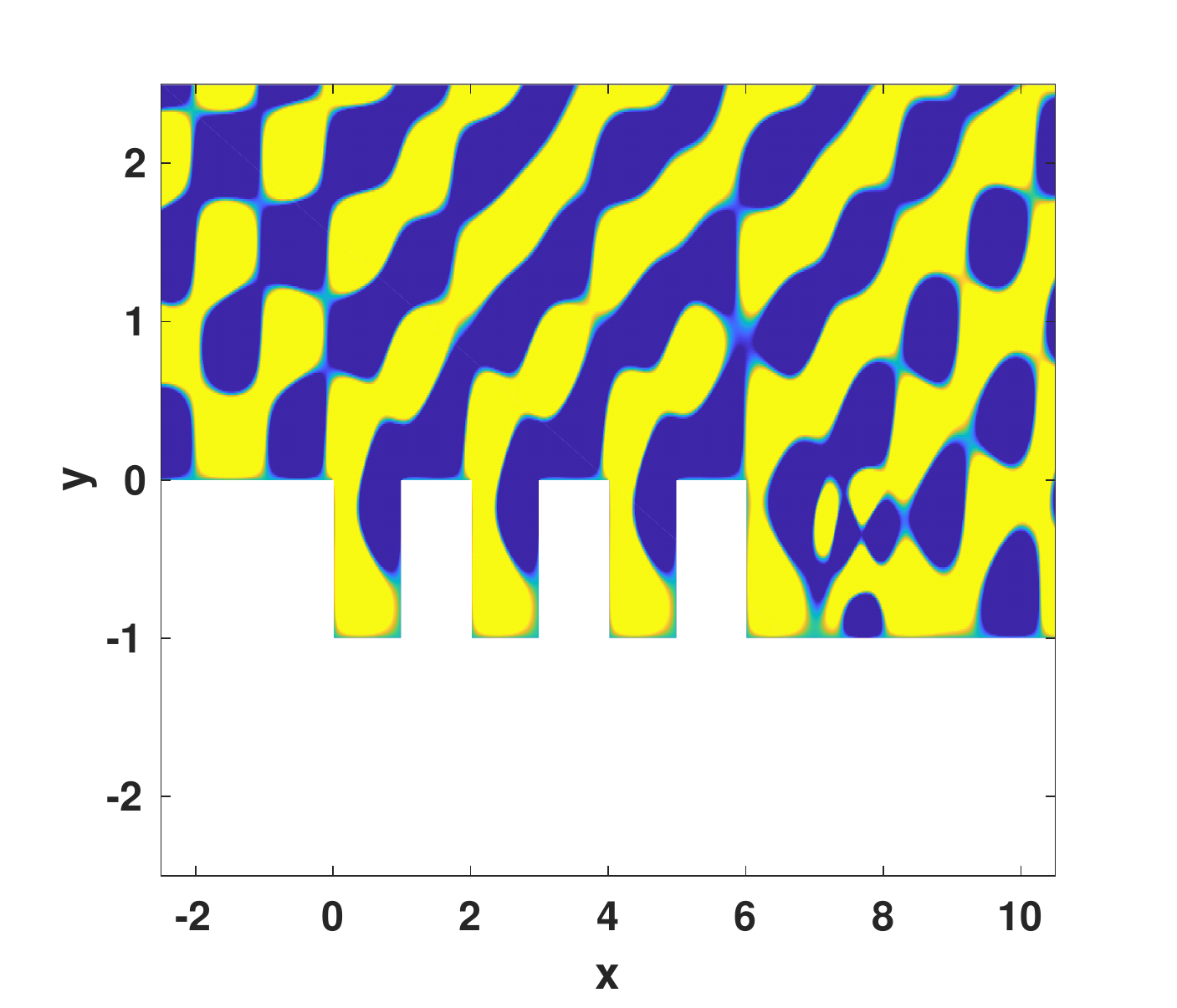}
  b)\includegraphics[width=0.277\textwidth]{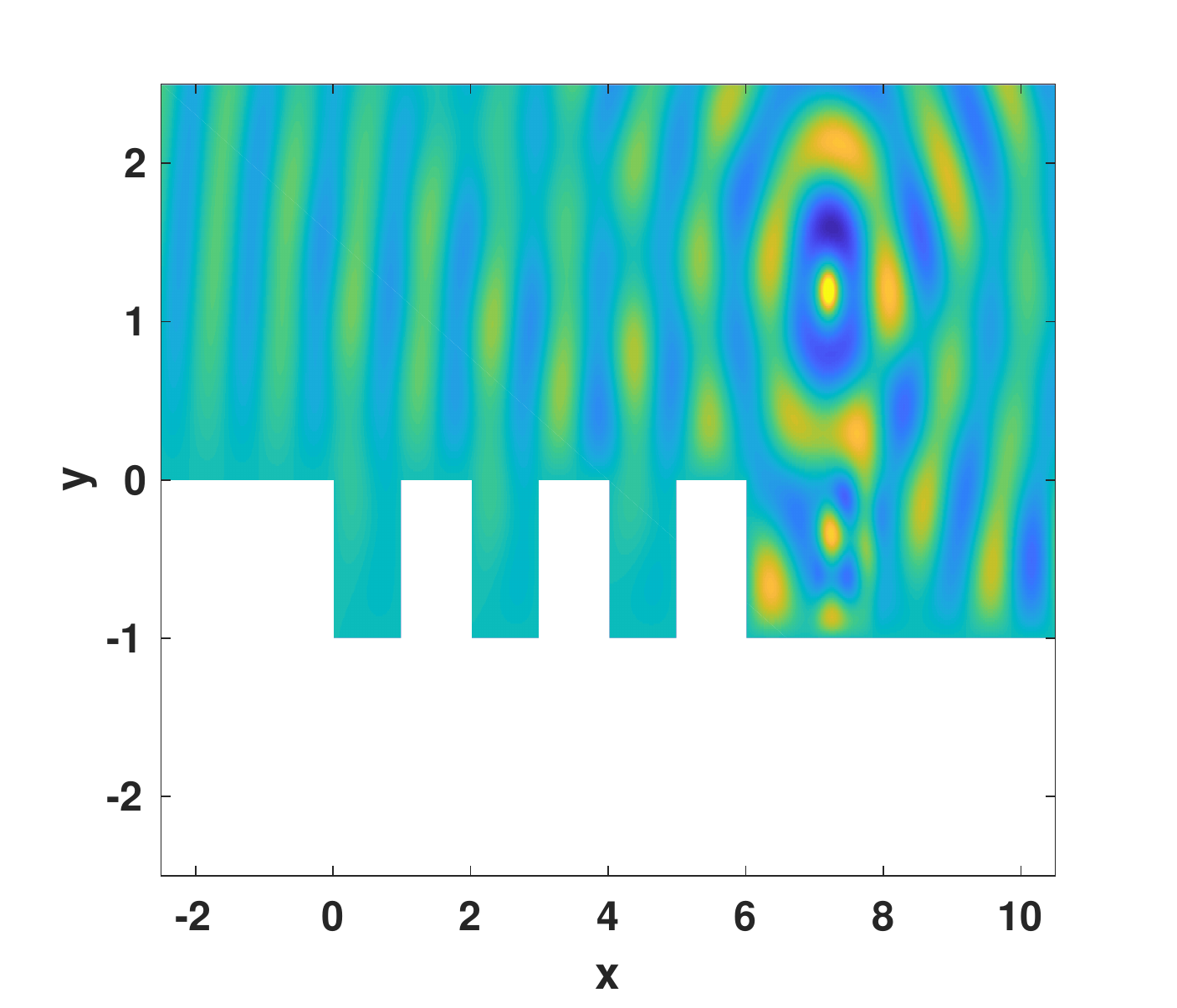}
  c)\includegraphics[width=0.327\textwidth]{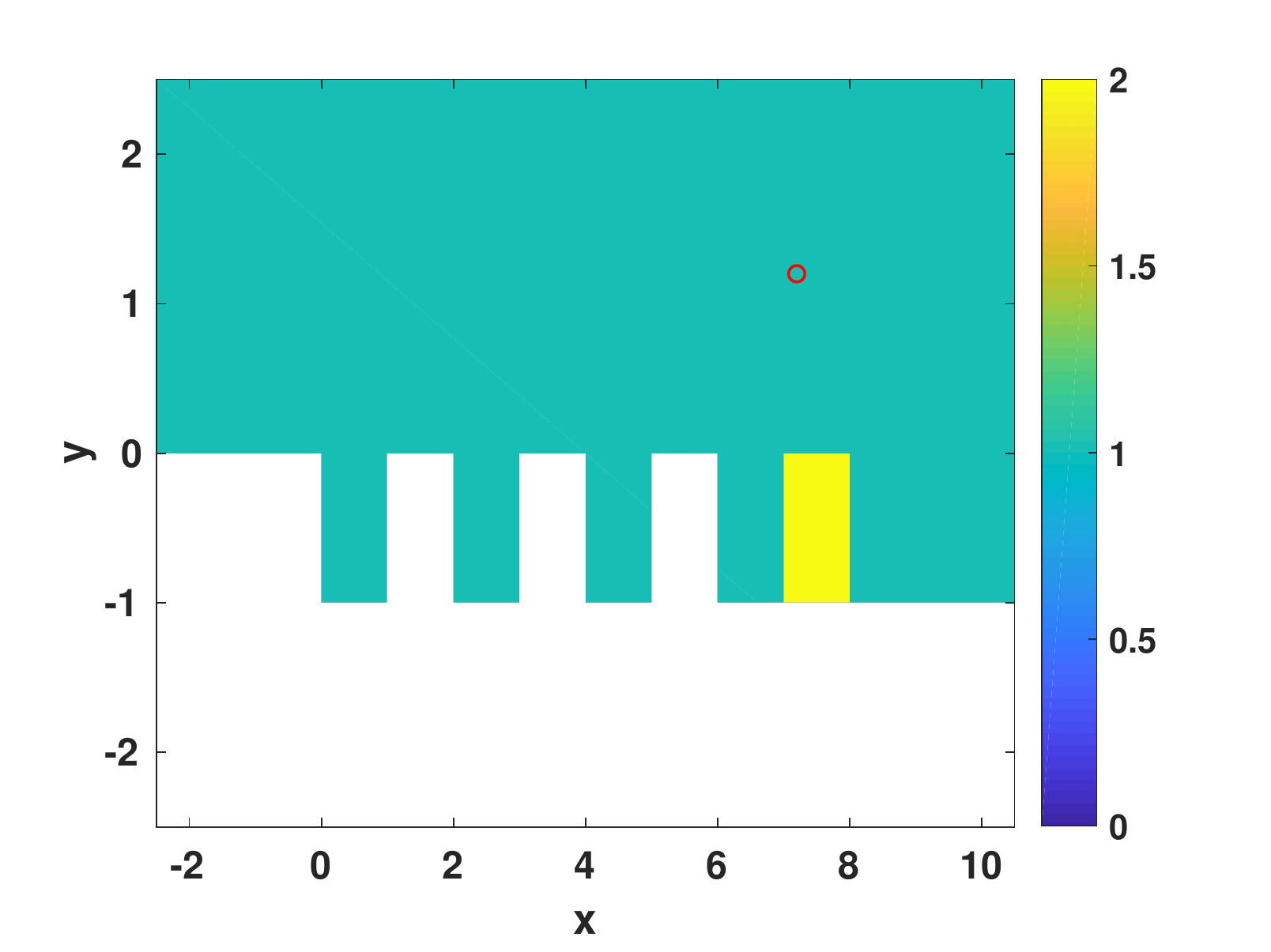}
	\caption{Real part of the total wave field $u^{tot}$ in
    $[-2.5,10.5]\times[-2.5,2.5]$ for: (a) incident plane wave with incident
    angle $\theta=\frac{\pi}{6}$; (b) incident cylindrical wave excited by the
    source point $(7.2,1.2)$. White region indicates the PEC substrate. (c)
    Profile of the refractive index above the PEC substrate indicated by the
    white region;  red circle indicates the location of the source
      point in (b)}.
  \label{fig:ex3:1}
\end{figure}

\section{Conclusion}

In this paper, we analyzed a {\color{rot} sound-soft rough surface scattering problem in two dimensions},
where the globally perturbed surface is assumed to consist of two horizontal
half lines pointing oppositely and one single vertical line segment connecting
their endpoints. For an incident plane wave, we enforced that the scattered wave,
post-subtracting reflected plane waves by the two half lines of the scattering
surface in their residential regions respectively, satisfies an integral form
of SRC condition (\ref{RC}) at infinity. With this new radiation condition, we
proved uniqueness and existence of weak solutions by a coupling scheme between
finite element and integral equation methods. Consequently, this indicates that
our new radiation condition is sharper than the ASR condition and UPRC, and generalizes
the radiation condition for scattering problems in a locally perturbed
half-plane.

Numerically, a NMM method was proposed based on our radiation condition. A
perfectly matched layer was setup to absorb the Sommerfeld-type outgoing waves.
Since the medium composes of two horizontally uniform regions, we expanded, in
either uniform region, the scattered wave in terms of eigenmodes and matched the
mode expansions on the common interface between the two uniform regions. This
leads to an algebraic linear system and thus yields numerical solutions to our
problem. Numerical experiments were carried out to validate the new radiation
condition and to show the performance of our numerical method. As the NMM method
is only limited to vertical and horizontal interfaces, a PML-based BIE method
will be developed in an ongoing work when the scattering surface
  contains more general curved interfaces. Besides, we shall extend the current
work to two-layer media with transmission interface conditions. In that case,
the substrates below the interface become penetrable and a new radiation
condition modeling the downward propagating waves should be established.

\section{Appendix: existence of the improper integral}
We need to show that
the following improper integral
\begin{equation}
  \label{eq:impint:1}
{\color{rot}   \lim_{R_1\to\infty}\int_{{\cal L}_{R_1}}e^{i(\alpha x_1 + \beta x_2)}\partial_{\bm \nu}\Phi(x;z)ds(x)= }\int_{\ML} \left[\nabla_x\, \Phi(x;z)\cdot \nu(x)\right] e^{i k\tau(\theta)\cdot x} ds(x),\quad z\in S_R,
\end{equation}
exists for a fixed incident angle $\theta\in(\varepsilon,\pi-\varepsilon)$ with
$\varepsilon>0$, where
$\Phi(x;z)$ is the background Green's function and
\begin{align*}
&\tau=\tau(\theta)=(\cos\theta,\sin\theta), \\
&\nu=\nu(\theta)=(\sin\theta,-\cos\theta), \\
&\ML=\ML(\theta)=\{x=s\tau(\theta),s_0\leq s<\infty\} \quad\mbox{with}\quad
s_0=R.
\end{align*}
Note that $\tau$ and $\nu$ denote respectively the tangential and normal directions at the ray $\ML\subset \mathbb R^2_+$. The integral (\ref{eq:impint:1})
is understood as the limit of
\[
\int_{\ML_l} \left[\nabla_x\, \Phi(x;z)\cdot \nu(x)\right] e^{i k\tau(\theta)\cdot x} ds(x),\quad \ML_l:=\{x=\tau(\theta)\, s: s_0\leq s\leq l\}
\] as  $l\rightarrow\infty$.
By \cite{S.P, CWEL, ZC98, ZC2003, Hu2018},
the two-dimensional background Green's function $\Phi$, which satisfies the {\color{rot}UASP}, can be written as
\[
\Phi(x;z)=G(x,z)-G(x,z_h^*)+ v(x; z),\quad x\in \Omega_\Gamma,
\] where $G$ is the free space's Green's function in $\R^2$ and
$z_h^*\in \mathbb R_-^2$ stands for the image of $z$ with respect to the reflection by the line
$\{z=(z_1,z_2)\in \R^2: z_2=-h\}$. The function $v(x; z)$ is a solution to the Helmholtz equation in $\Omega_\Gamma$.
Since $G(x,z)-G(x, z_h^*)$ decays with the order $|x_1|^{-3/2}$ on $\Gamma$,
it follows from (\cite{CWEL}) that the function $v$ belongs to the weighted Sobolev space $H_\varrho^1({\color{rot} \Omega_a })$ (see (\ref{WSS}) for the definition) for any $a\geq 0$ and $\varrho\in [0, 1)$. %, with
%$S_a:=\{x\in \Omega_\Gamma: x_2<a\}$ and
%\[
%H_\varrho^1(S_a):=\left\{u: \left[ \int_{S_a} \bigg( \big|
 %   (1+|x_1|^2)^{\varrho/2}u(x)\big|^2
%+ \left| \nabla
%  \left[(1+|x_1|^2)^{\varrho/2}u(x)\right]\right|^2\bigg)dx\right]^{1/2}<\infty
%\right\}.
%\]
 In view of the asymptotic behavior of the Hankel function for large arguments, one can readily prove that
\[
\partial_\nu G(x,z)=O(|x|^{-3/2}),\quad  \partial_\nu G(x,z^*_h)=O(|x|^{-3/2})\]
as $|x|\rightarrow \infty$ {\color{rot} on $\mathcal{L}$}, leading to
\begin{equation*}
 \int_{\ML} \nu\cdot \nabla_x\, (G(x,z)-G(x,z^*_h))\; e^{i k\tau\cdot x} ds<\infty.
\end{equation*}
Hence, it remains to consider the improper integral (\ref{eq:impint:1}) for $v$ in place of $\Phi$.

Recalling the Upward Propagating Radiation Condition (UPRC), we may represent $v$ as the integral
\[
v(x;z)=2\int_{y_2=0} \frac{\partial G(x,y)}{\partial y_2}\,v(y;z) dy_1,\quad x_2{\color{rot}>} 0.
\]
The improper integral in the above expression of $v$ can be understood as the duality between
$H_\varrho^{1/2}(\R)$ and its dual space $H_{-\varrho}^{-1/2}(\R)$ for any $\varrho\in [0, 1)$;
we refer to \cite{CWEL} for the equivalence of the UPRC and {\color{rot} UASR} in weighted Sobolev spaces. Obviously, \ben
\nu\cdot\nabla_x v(x;z)=2\int_{y_2=0} \left[\nu\cdot\nabla_x\frac{\partial G(x,y)}{\partial y_2}\right]\,v(y;z) dy_1,\quad x\in \ML,
\enn
where the first term in the integral can be written as
%\[
%\partial_\nu v(x)=O(|x|^{-3/2})\quad\mbox{as}\quad |x|\rightarrow\infty\quad\mbox{on}\quad L.
%\]
%We need the following ansatz,
%\begin{align}
%  \label{eq:ansatz:GL}
%  G^L(x,x^*) = ,
%\end{align}
%for $x$ above $\Gamma$, where and the bounded variation and absolutely continuous function $f(y_1,x_0)={\cal
%  O}(|y_1|^{-3/2})$ as $y_1\rightarrow\infty$ ({\color{rot} this statement
%  should be checked }!). By (\ref{eq:ansatz:GL}), we need to prove
%\begin{align}
%  \label{eq:impint:2}
%  \int_{L}ds\int_{y_2=0}(\nabla_x\frac{\partial G}{\partial y_2}(x,y)\cdot v)e^{i\tau\cdot x}f(y_1;x^*)dy_1
%\end{align}
%exists. Unfortunately, one cannot directly apply Fubini-Tonelli theorem on the
%space $L\times \{y_2=0\}$, since one may verify that the following integral
%\begin{align}
%  \label{eq:impint:3}
%  \int_{y_2=0}\left| f(y_1;x^*) \right|dy_1\int_{L}\left|  (\nabla_x\frac{\partial G}{\partial y_2}(x,y)\cdot \nu)e^{i\tau\cdot x}\right|ds
%\end{align}
%may not exist as an Lebesgue integral ({\color {rot} The reason is given at the
%  end of this section}). To resolve this issue, instead of regarding the
%integral over $L$ in (\ref{eq:impint:2}) as a Lebesgue integral, we understand
%it as an improper integral in the sense that
%\begin{align}
%\label{eq:impint:4}
%&\int_{L}ds\int_{y_2=0}(\nabla_x\frac{\partial G}{\partial y_2}(x,y)\cdot v)e^{i\tau\cdot x}f(y_1;x^*)dy_1 \nonumber\\
%=& \lim_{S\rightarrow \infty}\int_{L^S}ds\int_{y_2=0}(\nabla_x\frac{\partial G}{\partial y_2}(x,y)\cdot v)e^{i\tau\cdot x}f(y_1;x^*)dy_1,
%\end{align}
%where $L^S=\{x=\tau s, s_0\leq s\leq S\}$.
\begin{align}
\label{eq:dbder:G}
\nabla_x\frac{\partial G}{\partial y_2}(x,y)\cdot\nu =& \nabla_x\left( \frac{ik}{4}H_1^{(1)}(k|x-y|) \frac{x_2-y_2}{|x-y|}\right)\cdot\nu\nonumber\\
=&\frac{ik}{4}\Bigg(k{H_1^{(1)}}'(k|x-y|)\frac{x_2-y_2}{|x-y|^2}(x-y)\cdot \nu +
\frac{H_1^{(1)}(k|x-y|)}{|x-y|}\nu_2 \nonumber\\
&+ H_1^{(1)}(k|x-y|)(x_2-y_2)(-\frac{(x-y)\cdot {\color{rot} \nu }}{|x-y|^3})\Bigg).
\end{align}
For $y_2=0$ and $x=s \tau(\theta)\in \ML$, making use of  $x\cdot \nu(\theta)=0$ we obtain
\begin{align}
  \label{eq:dbder:G:2}
  \nabla_x\frac{\partial G}{\partial y_2}(x,y)\cdot\nu =-\frac{ik}{4}\Bigg[&k{H_1^{(1)}}'(k|x-y|)\frac{sy_1\sin^2\theta}{|x-y|^2} +
    \frac{H_1^{(1)}(k|x-y|)\cos\theta}{|x-y|} \nonumber\\
  &- H_1^{(1)}(k|x-y|)\frac{y_1s\sin^2\theta}{|x-y|^3}\Bigg].
\end{align}
For notational convenience, we write the distance $|x-y|$ in the previous relation as
\begin{align}
  \label{eq:nm:x-y}
  d(s,y_1):=|x-y| &= \sqrt{(s\cos\theta - y_1)^2 + s^2\sin^2\theta} \nonumber\\
  &= \sqrt{s^2 - 2s\cos\theta\, \tcr{y_1} + y_1^2} \nonumber \\
  &=\sqrt{(s-y_1\cos\theta)^2 + y_1^2\sin^2\theta}.
\end{align}
\tcr{This implies that
\be\label{inequality}
d(s,y_1)\geq s\sin\theta\geq s_0\sin\theta,\quad d(s,y_1)\geq |y_1|\,\sin\theta.
\en}
%For all $x\in \ML_l$ and for sufficiently large $|y_1|$, $|x-y|$ can be made
%^sufficiently large by (\ref{eq:nm:x-y}) so that there exists a constant $M>0$
%such that
\tcr{Hence, the right hand side of (\ref{eq:dbder:G:2}) can be bounded by}
\begin{align}
  \label{eq:dbder:G:3}
  |\nabla_x\frac{\partial G}{\partial y_2}(x,y)\cdot\nu|&\leq
\frac{Ms|y_1|\sin^2\theta}{|x-y|^{5/2}} + \frac{M|\cos\theta|}{|x-y|^{3/2}} +
\frac{Ms|y_1|\sin^2\theta}{|x-y|^{7/2}} \nonumber\\
&\leq \frac{M\,l}{|y_1|^{3/2}\sin^{1/2}\theta} + \frac{M|\cos\theta|}{|y_1|^{3/2}\sin^{3/2}\theta} + \frac{M\,\tcr{s}}{|y_1|^{5/2}\sin^{3/2}\theta},
\end{align}
\tcr{where the constant $M>0$ is uniform in all $y_1\in \R$.}
This indicates the finiteness of
\[
\int_{\ML_l}ds\int_{y_2=0}|(\nabla_x\frac{\partial G}{\partial y_2}(x,y)\cdot v)e^{i k \tau(\theta)\cdot x}|\,dy_1, \quad l\in \R^+.
\]
By Fubini's theorem, we can switch the order of integrations  to obtain
\begin{align}
  \label{eq:eq:impint:5}
  &\int_{\ML}  \partial_\nu v(x; z) e^{ik \tau(\theta){\color{rot}\cdot} x}\,d s(x)\\
  =&\int_{\ML}ds\int_{y_2=0}(\nabla_x\frac{\partial G}{\partial y_2}(x,y)\cdot \nu)e^{i k \tau(\theta)\cdot x}v(y_1,0;\,z)dy_1 \nonumber\\
  =& \lim_{l\rightarrow \infty}\int_{y_2=0}v({\color{rot}y_1},0;z)dy_1\int_{\ML_l}(\nabla_x\frac{\partial G}{\partial y_2}(x,y)\cdot \nu)e^{i{\color{rot}k\tau(\theta)}\cdot x}ds\nonumber\\
  =&\frac{-ik}{4}\lim_{l\rightarrow \infty}\int_{y_2=0}\,v(y_1,0;z)\left[  I_1^l(y_1) +I_2^l(y_1) + I_3^l(y_1)\right]dy_1,
\end{align}
where we have defined the following integrals (cf. (\ref{eq:dbder:G:2}))
\begin{align}
  \label{eq:def:I1}
  I_1^l(y_1) &:= \int_{s_0}^l \frac{ksy_1\sin^2\theta}{d^2(s,y_1)} { H_1^{(1)} }'(kd(s,y_1))e^{iks}ds,\\
  I_2^l(y_1) &:=\int_{s_0}^l \frac{\cos\theta}{d(s,y_1)} { H_1^{(1)} }(kd(s,y_1))e^{iks}ds,\\
  I_3^l(y_1) &:=-\int_{s_0}^l \frac{sy_1\sin^2\theta}{d^3(s,y_1)} { H_1^{(1)} }(kd(s,y_1))e^{iks}ds.
\end{align}
%In view of (\ref{inequality}), there exists a positive constant $M$
%such that for all $l\geq s_0$,
%Since $d(s,y_1)\geq s\sin\theta\geq s_0\sin\theta$ and since $d(s,y_1)\geq |y_1|\sin\theta$, we have for sufficiently
%large $s_0$ that
%\[
%  H_1^{(1)}(kd(s,y_1))\leq \frac{M}{d^{1/2}(s,y_1)}.
%\]
{\color{rot}As for $I_2^l$ and $I_3^l$, we can easily get the following estimates}
\begin{align}
  \label{eq:est:I2}
  |I_2^l(y_1)|&\leq \int_{s_0}^{l}\frac{M}{d^{3/2}(s,y_1)}ds\leq M\int_{s_0}^l s^{-3/2}\sin^{-3/2}\theta ds
  \leq 2M\sin^{-3/2}\theta s_0^{-1/2},
\end{align}
and
\begin{align}
  \label{eq:est:I3}
  |I_3^l(y_1)|&\leq \int_{s_0}^{l}\frac{Ms|y_1|\sin^2\theta}{d^{7/2}(s,y_1)}ds\leq \int_{s_0}^l \frac{M}{d^{3/2}(s,y_1)} ds
  \leq 2M\sin^{-3/2}\theta s_0^{-1/2},
\end{align} respectively.
Below we shall prove the boundedness of $I_1^l(y_1)$ for all $y_1\in \R$ and $l>s_0$.
For sufficiently large $s$, we can find a positive constant
$M$ such that
\[
  \left|{H_1^{(1)}}'(kd(s,y_1)) - {\color{rot}e^{-\pi i/4}}\left( \frac{2}{\pi k d(s,y_1)}\right)^{1/2}
      e^{ikd(s,y_1)} \right|\leq \frac{M}{d^{3/2}(s,y_1)} .
\]
Hence, by (\ref{eq:def:I1}),
\begin{align}
  \label{eq:est:I1}
  |I^l_1(y_1)| \leq &\frac{\sin^2\theta\sqrt{2k}}{\sqrt{\pi}}\left|  \int_{s_0}^{l}\frac{sy_1}{d^{5/2}(s,y_1)}e^{ik(d(s,y_1)+s)}ds\right|
                      + k\int_{s_0}^{l}\frac{Ms|y_1|\sin^2\theta}{d^{7/2}(s,y_1)} ds\nonumber\\
  \leq &\frac{\sin^2\theta\sqrt{2k}}{\sqrt{\pi}}\left|  \int_{s_0}^{l}\frac{sy_1}{d^{5/2}(s,y_1)}e^{ik(d(s,y_1)+s)}ds\right| + 2Mk\sin^{-3/2}\theta s_0^{-1/2}.
\end{align}
\tcr{The right hand side of (\ref{eq:est:I1}) is obviously bounded for $|y_1|\leq 1$, because of the first relation in (\ref{inequality}).
To derive an upper bound uniformly for $|y_1|>1$,}
we introduce a new variable $t(s)=d(s,y_1)+s$. Then one  can easily check that
\[
  t'(s)> 0,\quad s = \frac{t^2 - y_1^2}{2(t - y_1\cos\theta)},\quad t \geq
  |y_1|,\quad \frac{ds}{dt} = \frac{[d(t,y_1)]^2}{2(t-y_1\cos\theta)^2}.
\]
Thus, using change of variables we find
\begin{align}
  \label{eq:est:I1:1}
  \int_{s_0}^{l}\frac{sy_1}{d^{5/2}(s,y_1)}e^{ik(d(s,y_1)+s)}ds
 =& \int_{t(s_0)}^{t(l)}\frac{(t^2-y_1^2)y_1e^{ikt}}{d^3(t,y_1)(t-y_1\cos\theta)^{1/2}}\color{rot}{\sqrt{2}dt}\nonumber\\
=& \sin^2\theta\Bigg[  \left.\frac{y_1(t^2-y_1^2)e^{ikt}}{{\color{rot} i}kd^3(t,y_1)(t-y_1\cos\theta)^{1/2}}\right|_{t=t(s_0)}^{t=t(l)} \nonumber\\
  &- \int_{t(s_0)}^{t(l)}y_1\left( \frac{t^2-y_1^2}{d^3(t,y_1)(t-y_1\cos\theta)^{1/2}} \right)'\frac{e^{ikt}}{{\color{rot}i}k}dt\Bigg].
\end{align}
Since $t-y_1\cos\theta\geq |y_1|(1 -\cos\theta) = 2|y_1|\sin^2(\theta/2)$ and
 $d(t,y_1)\geq \max(t\sin\theta, |y_1|\sin\theta)$, we have
\begin{align}
  \label{eq:est:I1:2}
  \left|\left.\frac{y_1(t^2-y_1^2)e^{ikt}}{kd^3(t,y_1)(t-y_1\cos\theta)^{1/2}}\right|_{t=t(s_0)}^{t=t(l)}\right|\leq \frac{1}{k\sin^3\theta\sqrt{|y_1|}\sin(\theta/2)},
\end{align}
and
\begin{align}
&\left|  y_1\left( \frac{t^2-y_1^2}{d^3(t,y_1)(t-y_1\cos\theta)^{1/2}}
  \right)'\right|\nonumber\\
  =&\Bigg| \frac{2y_1t}{d^3(t,y_1)(u-y_1\cos\theta)^{1/2}} - \frac{3y_1(
  t^2-y_1^2 ) d'(t,y_1)}{d^4(t,y_1)(t-y_1\cos\theta)^{1/2}}
                  -\frac{y_1(t^2-y_1^2)}{2d^3(t,y_1)(t-y_1\cos\theta)^{3/2}}\Bigg|\nonumber\\
  \leq & \frac{2}{t\sin^3\theta(t-y_1\cos\theta)^{1/2}} + \frac{6}{t\sin^3\theta (t-y_1\cos\theta)^{1/2}} +\frac{1}{4\sin^2(\theta/2)t\sin^3\theta(t-y_1\cos\theta)^{1/2}}\nonumber \\
   \leq & \frac{{\color{rot} M}}{{\color{rot} \sin^2(\theta/2) }t\sin^3\theta(t-|y_1\cos\theta|)^{1/2}}.
\end{align}
Therefore,
\begin{align}
  \label{eq:est:I1:3}
  &\left|\int_{t(s_0)}^{t(l)}y_1\left( \frac{l^2-y_1^2}{d^3(t,y_1)(t-y_1\cos\theta)^{1/2}} \right)'\frac{e^{ikt}}{k}dt\right|\nonumber\\
  \leq & \int_{t(s_0)}^{t(l)}\frac{{\color{rot} M}}{kt{\color{rot}\sin^2(\theta/2)}\sin^3\theta(t-|y_1\cos\theta|)^{1/2}}dt\nonumber\\
  =&{\color{rot} \frac{M}{k\sin^3\theta\sin^2(\theta/2)} }\left\{
     \begin{array}{lc}
       2\left.\arctan\left( \sqrt{\frac{t-|y_1\cos\theta|}{|y_1\cos\theta|}} \right)\right|_{t=t(s_0)}^{t=t(l)}|y_1\cos\theta|^{-1/2}, & |y_1\cos\theta|\neq 0,\\
     \left.-\frac{1}{2\sqrt{t}}\right|_{t=t(s_0)}^{t=t(l)}, & |y_1\cos\theta| = 0,
     \end{array}
     \right.
\end{align}
which is bounded as well for sufficiently large $y_1$,  due to the fact $|\arctan x|\leq
\pi/2$ and  $t\geq t(s_0)\geq s_0$. Consequently, by (\ref{eq:est:I1}),
(\ref{eq:est:I1:1}), (\ref{eq:est:I1:2}) and (\ref{eq:est:I1:3}), we can
conclude that $|I_1^l(y_1)|$ is \tcr{uniformly bounded for all $y_1\in \R$}.
This together with boundedness of $I_2^{l}$ and $I_3^{l}$ implies that
$I_j^l\in L_{-\varrho}(\R)$ with $\varrho>1/2$ and thus lies in the dual space
$H^{-1/2}_{-\varrho}(\R)$ of $v(y_1,0;z)$ for $\varrho\in(1/2, 1)$.
By the dominated
convergence theorem, we can now claim that the integral in
(\ref{eq:eq:impint:5}) exists. This proves the boundedness of the improper integral (\ref{eq:impint:1}).

%\section{Why (\ref{eq:impint:3}) may not exist}
%{\color{rot}
%We have
%\begin{align}
%  \label{eq:est:1}
% &\int_{L} \left|  {H_1^{(1)}}'(k|x-y|)\frac{x_2y_1\sin\theta}{|x-y|^2}\right|ds\nonumber\\
%  =&\int_{s_0}^{+\infty}\frac{s|y_1|\sin^2\theta|{H_1^{(1)}}'(k\sqrt{(s-y_1\cos\theta)^2 + y_1^2\sin^2\theta})|}{(s-y_1\cos\theta)^2+y_1^2\sin^2\theta}  ds\nonumber\\
%  =&\int^{(s_0-y_1\cos\theta)^2+y_1^2\sin^2\theta}_{y_1^2\sin^2\theta} + \int_{y_1^2\sin^2\theta}^{+\infty}\frac{|y_1|\sin^2\theta |{H_1^{(1)}}'(k\sqrt{u})|}{2u} du\nonumber\\
%  &+\int_{s_0}^{\infty}\frac{y_1\cos\theta|y_1|\sin^2\theta|{H_1^{(1)}}'(k\sqrt{(s-y_1\cos\theta)^2 + y_1^2\sin^2\theta})|}{(s-y_1\cos\theta)^2+y_1^2\sin^2\theta}  ds.
%\end{align}
%Considering a special case $\theta=\pi/2$ for normal incidence, we find
%\begin{equation}
%  \int_{L} \left|  {H_1^{(1)}}'(k|x-y|)\frac{x_2y_1}{|x-y|^2}\right|ds\\
%  \leq \int_{y_1^2}^{\infty}\frac{|y_1||{H_1^{(1)}}'(k\sqrt{u})|}{u} du.
%\end{equation}
%For sufficiently large $|y_1|$, there exists $M>0$  such that
%\[
%  |{H_1^{(1)}}'(k\sqrt{u})|\leq M u^{-1/4}.
%\]
%Thus, we have that
%\[
%\int_{L} \left|
%  {H_1^{(1)}}'(k|x-y|)\frac{x_2y_1}{|x-y|^2}\right|ds
%\leq M \int_{y_1^2}^{\infty} |y_1| u^{-5/4} du=
%4M|y_1|\left.u^{-1/4}\right|_{u=y_1^2}^{\infty} = 4M\sqrt{|y_1|}.
%\]
%As $|y_1|\rightarrow\infty$,
%\[
%|f(y_1;x^*)|\int_{L} \left|
%  {H_1^{(1)}}'(k|x-y|)\frac{x_2y_1}{|x-y|^2}\right|ds \sim \frac{4M}{|y_1|}.
%\]
%one cannot gaurentee the existence of
%\[
%\int_{y_2=0}|f(y_1;x^*)|\int_{L} \left|
%  {H_1^{(1)}}'(k|x-y|)\frac{x_2y_1}{|x-y|^2}\right|ds dy_1.
%\]
%}

\section*{Acknowledgement}

WL thanks Prof. Ya Yan Lu for useful discussion. The work of G. Hu is
supported by the NSFC grant (No. 11671028) and NSAF grant (No. U1530401).
%{\color{rot} The authors would like to thank the anonymous referees for their comments and suggestions which help improve the original version of this paper.}


\begin{thebibliography}{99}

\bibitem{Bao00} H.~Ammari, G.~Bao and A.~Wood. \newblock An integral equation method for the electromagnetic
 scattering from cavities. \newblock{\em Math. Meth. Appl. Sci.}, 23: 1057-1072, 2000.


 \bibitem{baohuyin18}
G.~Bao, G.~Hu and T.~Yin.
\newblock Time-harmonic acoustic scattering from locally perturbed half-planes.
\newblock {\em: SIAM J. Appl. Math.}, {\color{rot} 78 (5): 2672-2691, 2018}.

\bibitem{ber94}
J.~P. Berenger.
\newblock A perfectly matched layer for the absorption of electromagnetic
  waves.
\newblock {\em J. Comput. Phys.}, 114: 185--200, 1994.

\bibitem{biebae01}
P.~Bienstman and R.~Baets.
\newblock Optical modelling of photonic crystals and vcsels using eigenmode
  expansion and perfectly matched layers.
\newblock {\em Opt. Quant. Electron.}, 33: 327--341, 2016.

\bibitem{biederbaeolydez01}
P.~Bienstman, H.~Derudder, R.~Baets, F.~Olyslager, and D.~De~Zutter.
\newblock Analysis of cylindrical waveguide discontinuities using vectorial
  eigenmodes and perfectly matched layers.
\newblock {\em {IEEE} {T}rans. {M}icrow. {T}heory {T}ech.}, 49:349--354, 2001.

\bibitem{botcramcp81}
L.C. Botten, M.S. Craig, and R.C. McPhedran.
\newblock Highly conducting lamellar diffraction gratins.
\newblock {\em Optica Acta}, 28:1103--1106, 1981.

\bibitem{brulyoperaratur16}
O.~P. Bruno, M.~Lyon, C.~P\'{e}rez-Arancibia, and C.~Turc.
\newblock Windowed {G}reen function method for layered-media scattering.
\newblock {\em SIAM Journal on Applied Mathematics}, 76(5):1871--1898, 2016.


\bibitem{cai02}
W.~Cai.
\newblock Algorithmic issues for electromagnetic scattering in layered media:
  {G}reen's functions, current basis, and fast solver.
\newblock {\em Advances in Computational Mathematics}, 16(2):157--174, 2002.


 \bibitem{S.P} S.~N.~Chandler-Wilde and P.~Monk. \newblock Existence, uniqueness and
 variational methods for scattering by unbounded rough surfaces. \newblock{\em
 SIAM J. Math. Anal.},  37:598--618, 2005.

\bibitem{CWEL} S.~N.~Chandler-Wilde and J.~Elschner. \newblock
 Variational approach in weighted Sobolev spaces to scattering by unbounded rough surfaces. \newblock{\em
SIAM J. Math. Anal.}, 42:2554--2580, 2010.


 \bibitem{CRZ1998} S.~N. Chandler-Wilde, C.~Ross and B.~Zhang,  {\it Scattering by rough surfaces},
 in Proceedings of the Fourth International Conference on Mathematical
and Numerical Aspects of Wave Propagation (J. DeSanto, ed.), 164-168,
SIAM, 1998.

\bibitem{CHP2006} S.~N. Chandler-Wilde, E.~Heinemeyer and R.~Potthast, \newblock A well-posed integral equation formulation for three-dimensional rough surface scattering. \newblock {\em Proc. R. Soc. A}, 462: 3683--3705, 2006.




\bibitem{chezhe10}
Z.~Chen and W.~Zheng.
\newblock Convergence of the uniaxial perfectly matched layer method for
  time-harmonic scattering problems in two-layered media.
\newblock {\em SIAM J. Numer. Analy.}, 48:2158--2185, 2010.

\bibitem{chew95}
W.~C. Chew.
\newblock {\em Waves and fields in inhomogeneous media}.
\newblock IEEE PRESS, New York, 1995.

\bibitem{chewee94}
W.~C. Chew and W.~H. Weedon.
\newblock A 3{D} perfectly matched medium for modified {M}axwell's equations
  with stretched coordinates.
\newblock {\em Microwave and Optical Technology Letters}, 7(13):599--604, 1994.

\bibitem{chiyehshi09}
Y.-P. Chiou, W.-L. Yeh, and N.-Y. Shih.
\newblock Analysis of highly conducting lamellar gratings with multidomain
  pseudospectral method.
\newblock {\em J. Lightwave Technol.}, 27:5151--5159, 2009.

\bibitem{colkre13}
D.~Colton and R.~Kress.
\newblock {\em {Inverse {A}coustic and {E}lectromagnetic {S}cattering {T}heory
  (3rd {E}dition)}}.
\newblock Springer, 2013.


\bibitem{J.A} J.~A.~De Santo. \newblock {\em Scattering by rough surfaces}, in
 Scattering: Scattering and Inverse Scattering in Pure and Applied
 Science, R. Pike and P. Sabatier, eds, Academic Press, New
 York, (2002), pp. 15-36.


\bibitem{derdezoly98}
H.~Derudder, De~Zutter D., and F.~Olyslager.
\newblock Analysis of waveguide discontinuities using perfectly matched layers.
\newblock {\em Electron. Lett.}, 34:2138--2140, 1998.



\bibitem{gra99}
G.~Granet.
\newblock Reformulation of the lamellar grating problem through the concept of
  adaptive spatial resolution.
\newblock {\em J. Opt. Soc. Am. A}, 16:2510--2516, 1999.

\bibitem{gragui96}
G.~Granet and B.~Guizal.
\newblock Efficient implementation of the coupled-wave method for mettalic
  lamellar gratings in tm polarization.
\newblock {\em J. Opt. Soc. Am. A}, 13:1019--1023, 1996.


\bibitem{Hu2018} G.~Hu and A.~Rathsfeld. \newblock Acoustic scattering from locally perturbed sound-soft periodic surfaces. \newblock WIAS Preprint No. 2522, 2018.

%\bibitem{ZC1998} B. Zhang and S. N. Chandler-Wilde,
% Acoustic scattering by an inhomogeneous layer on a rigid
%plate, SIAM J. Appl. Math., 58 (1998), pp. 1931-1950.



\bibitem{Hsiao} G.~C.~Hsiao and W.~L.~Wendland. \newblock{\em Boundary Integral Equations}. \newblock Springer-Verlag, Berlin, 2008.

  \bibitem{Jin98}
J.~M.~Jin. \newblock Electromagnetic scattering from large, deep, and arbitrarily-shaped open cavities.
\newblock{ Electromagnetics}, 18: 3--34, 1998.




\bibitem{kno78}
K.~Knop.
\newblock Rigorous diffraction theory for transmission phase gratings with deep
  rectangular grooves.
\newblock {\em J. Opt. Soc. Am.}, 68:1206--1210, 1978.



\bibitem{laigreone16}
J.~Lai, L.~Greengard, and M.~O\'Neil.
\newblock A new hybrid integral representation for frequency domain scattering
  in layered media.
\newblock {\em Appl. Comput. Harmon. Anal.}, in press, 2016.

\bibitem{lalmor96}
P.~Lalane and G.~M. Morris.
\newblock Highly improved convergence of the coupled-wave method for tm
  polarization.
\newblock {\em J. Opt. Soc. Am. A}, 13:779--784, 1996.



\bibitem{Li2010}
P.~Li. \newblock Coupling of finite element and boundary integral methods for electromagnetic scattering
in a two-layered medium. \newblock {\em J. Comput. Phys.}, 229: 481-497, 2010.



\bibitem{li93a}
L.~Li.
\newblock A modal analysis of lamellar diffraction gratings in conical
  mountings.
\newblock {\em J. Mod. Opt.}, 40:553--573, 1993.

\bibitem{li96}
L.~Li.
\newblock Use of Fourier series in the analysis of discontinuous periodic
  structures.
\newblock {\em J. Opt. Soc. Am. A}, 13: 1870--1876, 1996.

\bibitem{luluqia18}
W.~Lu, Y.~Y. Lu, and J.~Qian.
\newblock Perfectly matched layer boundary integral equation method for wave
  scattering in a layered medium.
\newblock {\em SIAM J. Appl. Math.}, 78(1): 246--265, 2018.

\bibitem{luluson18}
W.~Lu, Y.~Y. Lu, and D.~Song.
\newblock A numerical mode matching method for wave scattering in a layered
  medium with a stratified inhomogeneity.
\newblock {\em submitted}, 2018.

\bibitem{lushilu14}
X.~Lu, H.~Shi, and Y.~Y. Lu.
\newblock Vertical mode expansion method for transmission of light through a
  circular hole in a slab.
\newblock {\em J. Opt. Soc. Am. A}, 31: 293--300, 2014.

%\bibitem{Mclean}
%W.~Mclean. \newblock {\em Strongly elliptic systems and boundary integral equations}, \newblock Cambridge University
% Press, Cambridge, 2000.

\bibitem{mon03}
P.~Monk.
\newblock {\em Finite Element Methods for Maxwell's Equations}.
\newblock Oxford University Press, 2003.

\bibitem{mor95}
R.~H. Morf.
\newblock Exponentially convergent and numerically efficient solution of
  {M}axwell's equations for lamellar gratings.
\newblock {\em J. Opt. Soc. Am. A}, 12: 1043--1056, 1995.

\bibitem{roazha92}
G.~F. Roach and B.~Zhang.
\newblock The limiting-amplitude principle for the wave propagation problem
  with two unbounded media.
\newblock {\em Math. Proc. Cambridge Philos. Soc.}, 112: 207--223, 1992.

\bibitem{shestesan82}
P.~Sheng, R.~S. Stepleman, and P.~N. Sanda.
\newblock Exact eigenfunctions for square wave gratings -- application to
  diffraction and surface-plasmon calculations.
\newblock {\em Phys. Rev. B}, 26: 2907--2916, 1982.

\bibitem{sonyualu11}
D.~Song, L.~Yuan, and Y.~Y. Lu.
\newblock Fourier-matching pseudospectral modal method for diffraction
  gratings.
\newblock {\em J. Opt. Soc. Am. A}, 28: 613--620, 2011.



\bibitem{Willers1987} A.~Willers. \newblock The helmholtz equation in disturbed half-spaces. \newblock{ Math. Meth. Appl. Sci.}, 9: 312--323, 1987


 \bibitem{Wood2006} A.~Wood. \newblock Analysis of electromagnetic scattering from an overfilled cavity in the ground plane. \newblock{J. Comput. Phys.}, 215:, 630--641, 2006.


\bibitem{ZC98} B.~Zhang and S.~N. Chandler-Wilde.
 \newblock A uniqueness result for scattering by infinite rough surfaces. \newblock {\em SIAM J. Appl. Math.}, 58: 1774--1790, 1998.

\bibitem{ZC2003}
 B.~Zhang and Simon N. ~Chandler-Wilde. \newblock Integral equation methods for scattering by infinite rough surfaces. \newblock{\em Math. Methods Appl. Sci.}, 26: 463--488, 2003.

\end{thebibliography}
\end{document}